\documentclass[11pt,leqno]{amsart}
\usepackage{filecontents}
\usepackage{microtype}
\usepackage[english]{babel}
\usepackage[T1]{fontenc}
\usepackage[utf8]{inputenc}
\usepackage[T1]{fontenc}
\usepackage{lmodern}
\usepackage{latexsym}
\usepackage{verbatim}
\usepackage{hyperref}
\usepackage{amsthm}
\usepackage{amsmath}
\usepackage{amsfonts}
\usepackage[italian]{varioref}
\usepackage{imakeidx}
\usepackage{amssymb}
\usepackage{xcolor}
\usepackage{esint}
\usepackage{accents}
\usepackage{enumitem}
\usepackage{caption} 

\usepackage{tikz}
\usepackage{tikz,tikz-3dplot}
\usepackage{pgfplots}
\pgfplotsset{compat=1.16}
\tikzset{
    cross/.pic = {
    \draw[rotate = 45] (-#1,0) -- (#1,0);
    \draw[rotate = 45] (0,-#1) -- (0, #1);
    }
}
\usetikzlibrary{calc, angles}
\usetikzlibrary{arrows}
\usetikzlibrary{positioning}
\usetikzlibrary{intersections}
\usepackage{pgfplots}
\pgfplotsset{compat=newest}
\usepgfplotslibrary{fillbetween}
\usepackage{pgfplots}
\pgfplotsset{compat=newest}
\usepgfplotslibrary{fillbetween}
\usetikzlibrary{patterns}

	\definecolor{ao(english)}{rgb}{0.0, 0.5, 0.0}

\usepackage{lipsum}
\usepackage{curve2e}
\definecolor{gray}{gray}{0.4}

\usepackage{amsmath}
\usepackage{amsfonts}
\usepackage{amssymb}
\usepackage{graphicx}
\usepackage{hyperref} 
\usepackage{mathrsfs} 
\usepackage{relsize} 

\theoremstyle{plain}
\newtheorem{theorem}{Theorem}[section]

\newtheorem{corollary}[theorem]{Corollary}
\newtheorem{definition}[theorem]{Definition}
\newtheorem{lemma}[theorem]{Lemma}
\newtheorem*{problem*}{Problem}

\theoremstyle{definition}

\newtheorem{remark}[theorem]{Remark}

\def\Item$#1${\item $\displaystyle#1$
   \hfill\refstepcounter{equation}(\theequation)}
   
\theoremstyle{plain} 

\theoremstyle{definition}

\theoremstyle{remark} 
\numberwithin{equation}{section}
\usepackage[makeroom]{cancel}

\usepackage{bm} 
\usepackage{graphicx}
\makeindex
\usepackage{mathtools}
\usepackage{slashed}
\date{}

\usepackage{tikz,tikz-3dplot}
\tdplotsetmaincoords{80}{45}
\tdplotsetrotatedcoords{-90}{180}{-90}
\usetikzlibrary{arrows, automata, intersections}

\newcommand{\rr}{\mathbb{R}}

\newcommand{\sss}{\mathbb{S}}
\newcommand{\e}{\epsilon}

\newcommand{\spc}{\ \ \ \ }

\newcommand{\inn}{\textnormal{in}\ }
\newcommand{\onn}{\textnormal{on}\ }
\newcommand{\andd}{\textnormal{and}}
\newcommand{\inte}{\textnormal{int}\ }

\newcommand{\dpar}{\mathcal{D}\textnormal{-parabolic}}

\newcommand{\bra}[1]{\{#1\}}
\newcommand{\loc}{_{loc}}

\newcommand{\system}[2]{\left\{\begin{array}{#1} #2 \end{array} \right.}
\newcommand{\ric}{\textnormal{Ric}}

\newcommand{\lap}{\Delta}
\newcommand{\dint}[1]{\ \textnormal{d}#1}
\newcommand{\ubra}[1]{^{#1}}
\newcommand{\dbra}[1]{_{#1}}
\newcommand{\norm}[2]{\left|\left| #1 \right| \right|_{#2}}
\newcommand{\dive}[1]{\textnormal{div}\left( #1 \right)}
\newcommand{\operl}{\mathcal{L}}
\newcommand{\operm}{\mathcal{M}}
\newcommand{\dvol}{\ \textnormal{dv}}

\tikzstyle{mybox} = [draw=black, very thick, rectangle, rounded corners, inner ysep=5pt, inner xsep=5pt]

\begin{document}

\begin{abstract}
The necessity of a Maximum Principle arises naturally when one is interested in the study of qualitative properties of solutions to partial differential equations. In general, to ensure the validity of these kinds of principles one has to consider some additional assumptions on the ambient manifold or on the differential operator. The present work aims to address, using both of these approaches, the problem of proving Maximum Principles for second order, elliptic operators acting on unbounded Riemannian domains under Dirichlet boundary conditions. Hence there is a natural division of this article in two distinct and standalone sections.
\end{abstract}

\title{Maximum principles in unbounded Riemannian domains}
\author{Andrea Bisterzo}
\address{Universit\`a degli Studi di Milano-Bicocca\\ Dipartimento di Matematica e Applicazioni \\ Via Cozzi 55, 20126 Milano - ITALY}
\email{a.bisterzo@campus.unimib.it}
\maketitle


\section{Introduction}
In this work we address the validity of the maximum principle for bounded solutions to the problem
\begin{align*}
\system{ll}{\Delta u\geq c u & \inn \Omega \\ u \leq 0 & \onn \partial \Omega}
\end{align*}
where $\Omega$ is an unbounded domain inside the Riemannian manifold $(M,g)$. We shall present two kinds of results where the common root is the assumption that $\Omega$ is ``small'' from the viewpoint of the operator. The first result requires that the underlying manifold has a special structure (warped product cylinder) and the smallness of the domain is encoded in its (Dirichlet) parabolicity. The second result has a more abstract flavour as it holds in any Riemannian manifold provided that the domain is small in a spectral sense.


In the Euclidean setting a classical Maximum Principle for unbounded domains contained in the complement of a cone states as follows (for a reference, see \cite{berestycki1997monotonicity})

\begin{theorem}\label{Thm:EuclideanMaximumPrinciple}
Consider a possibly unbounded domain $\Omega \subset \rr^n$, $n\geq 2$, whose closure is contained in the complement of a non-degenerate solid cone $\mathcal{C}\subset \rr^n$. If $u\in C^0(\overline{\Omega})\cap W^{1,2}\loc (\Omega)$ is a distributional solution to
\begin{align*}
\system{ll}{-\lap u + c\ u \leq 0 & in\ \Omega \\ u\leq 0 & on\ \partial \Omega \\ \sup_\Omega u <+\infty,}
\end{align*}
where $0\leq c\in C^0(\Omega)$, then
\begin{align*}
u\leq 0 \spc in\ \Omega.
\end{align*}
\end{theorem}

The proof is essentially based on the fact that the Euclidean space is a model manifold, that is, the manifold obtained by quotienting the warped product $([0,+\infty)\times \sss^{n-1}, \textnormal{d}r\otimes \textnormal{d}r + r^2 g^{\sss^{n-1}})$ with respect to the relation that identifies $\{0\}\times \sss^{n-1}$ with a point $o$, called \textit{pole}, and then extending smoothly the metric in $o$.

Influenced by the model structure of $\rr^n$, in Section \ref{Sec:Maximum principle and cone} we obtain a transposition of the previous theorem to warped product manifolds satisfying certain (radial) curvature conditions and replacing the notion of \textit{cone} with the notion of \textit{strip}. The assumptions on the geometry of $M$ and on $\Omega$ are needed to construct a suitable barrier function, crucial for the validity of the result. We stress that the main theorem of Section \ref{Sec:Maximum principle and cone} will be first stated in the context of (Dirichlet-)parabolic manifolds and then reinterpreted in the language of maximum principles. This is the content of Corollary \ref{Cor:Maximum principle cone}.


On the other hand, if we want to recover a maximum principle without requiring any assumption on the structure of the manifold (and of the domain), then we have to consider some additional hypotheses on the differential operator and on its spectrum. These kinds of assumptions are natural if one compares with the compact case.

\begin{theorem}
Let $(M,g)$ a Riemannian manifold, $\Omega\subseteq M$ a bounded domain and $\mathcal{L}$ a linear elliptic operator with (sufficiently) regular coefficients. Then, the Maximum Principle holds for $\mathcal{L}$ in $\Omega$ with Dirichlet boundary conditions if and only if the first Dirichlet eigenvalue of $\mathcal{L}$ on $\Omega$ is positive.
\end{theorem}

Inspired by this fact, one might wonder if this property can be generalized to unbounded domains. This is true in the Euclidean space according to the very interesting work \cite{nordmann2021maximum} by Samuel Nordmann. In Section \ref{Sec:ABP and Maximum Principle} we shall extend Nordmann result to Riemannian domains.

To this end, we first obtain an ABP-like inequality for the differential operator $\mathcal{L}$ acting on bounded smooth domains. Next, we will use it to construct a couple of generalized eigenelements $(\lambda_1, \varphi)$ for $\mathcal{L}$ on possibly nonsmooth bounded domains and, using an exhaustion argument, on unbounded smooth domains. Following the proof obtained by Nordmann, in Theorem \ref{Thm:UnboundedMaximumPrinciple} we get a maximum principle for the operator $\mathcal{L}$ acting on an unbounded smooth domain $\Omega$ of a general Riemannian manifold $(M,g)$ under the assumption that $\lambda_1>0$.

In the last section we will apply Theorem \ref{Thm:UnboundedMaximumPrinciple} to generalize some of the results obtained in \cite{bisterzo2022symmetry} by the author together with Stefano Pigola.

\section{Maximum principle for unbounded domains in the complement of a strip}\label{Sec:Maximum principle and cone}

The already cited Theorem \ref{Thm:EuclideanMaximumPrinciple} is a milestone in the Euclidean analysis of PDEs. A possible proof makes use of the next classical lemma (see \cite[Lemma 2.1]{berestycki1997monotonicity}), which is based on the existence of a suitable positive $(-\Delta+c)$-subharmonic function. We state this result in a more general setting.

\begin{lemma}\label{Lem:BCNLemma}
Let $(M,g)$ be a complete manifold. Given a (possibly unbounded) domain $\Omega \subset M$, suppose $u\in W^{1,2}_{\loc}(\Omega)\cap C^0(\overline{\Omega})$ is a distributional solution to
\begin{align*}
\system{ll}{-\Delta u+c\ u\leq 0 & in\ \Omega \\ u\leq 0 & on\ \partial \Omega \\ \sup_\Omega u<+\infty ,}
\end{align*}
where $0\leq c\in C^0(\Omega)$. If there exists a function $\phi \in C^2(\Omega)\cap C^0 (\overline{\Omega})$ (possibly depending on $u$) satisfying
\begin{align*}
\system{ll}{-\Delta \phi + c\ \phi \geq 0 & in\ \Omega \\ \phi>0 & in\ \overline{\Omega}}
\end{align*}
and
\begin{align*}
\underset{p\in \Omega}{\limsup_{d^M (p,p_0)\to +\infty,}}  \ \frac{u(p)}{\phi(p)}\leq 0
\end{align*}
for any fixed $p_0\in \Omega$ (where $d^M$ is the intrinsic distance on $M$), then $u\leq 0$ in $\Omega$.
\end{lemma}

\begin{proof}
Let $w:=\frac{u}{\phi}\in W^{1,2}\loc(\Omega)\cap C^0(\overline{\Omega})$. We have
\begin{align*}
\lap w + 2 g\left(\nabla w, \frac{\nabla \phi}{\phi} \right)+ w \frac{\lap \phi}{\phi}= \frac{\lap u}{\phi}\geq c\ \frac{u}{\phi}=c\ w \spc \spc \inn \mathcal{D}'
\end{align*}
i.e.
\begin{align*}
\mathcal{L}w:= -\lap w-2g\left(\nabla w, \frac{\nabla \phi}{\phi} \right)+ w \frac{-\lap \phi + c\ \phi}{\phi}\leq 0 \spc \spc \inn \mathcal{D}'.
\end{align*}
By assumption, for any $\e>0$ and any fixed $p_0\in M$ there exists $0<R_\e\xrightarrow{\e \to 0} \infty$ so that $w(p)\leq \e$ for every $p\in \Omega$ satisfying $d^M (p,p_0)\geq R_\e$. Hence, for $\Omega_\e:=B^{M}_{R_\e}(p_0)\cap \Omega$ we get
\begin{align*}
\system{rl}{\mathcal{L}w\leq 0 & \textnormal{in any connected component of}\ \Omega_\e\\
w\leq \e & \textnormal{on the boundary of any connected component of}\ \Omega_\e.}
\end{align*}
Since $\frac{-\lap \phi + c \phi}{\phi}\geq 0$, by the standard maximum principle $w\leq \e$ in any connected component of $\Omega_\e$. Letting $\e\to 0$ we get $w\leq 0$ in $\Omega$, i.e. $u\leq 0$ in $\Omega$.
\end{proof}

As said above, the previous lemma is the key ingredient to obtain the unbounded maximum principle contained in Theorem \ref{Thm:EuclideanMaximumPrinciple}. Indeed, for any bounded above supersolution $u$ we only have to find a barrier function $\phi$ satisfying the assumptions of Lemma \ref{Lem:BCNLemma}. Observe that, since in Theorem \ref{Thm:EuclideanMaximumPrinciple} $u$ is assumed to be bounded above, the dependence of $\phi$ on $u$ may be bypassed just requiring that $\phi\xrightarrow[]{|x|\to +\infty}+\infty$.

It is precisely the presence of the cone $\mathcal{C}$ in the complement of $\Omega$ that allows us to easily construct $\phi$. 

\begin{proof}[Proof of Theorem \ref{Thm:EuclideanMaximumPrinciple}]
Consider the spherical coordinates $(r,\theta)$ on $\rr^n$ and set $\Lambda=\sss^{n-1} \setminus \mathcal{C}$. We define $\phi$ as the restriction to $\Omega$ of the function $\Phi:(0,+\infty)\times \Lambda\to \rr_{\geq 0}$ given by
\begin{align*}
\Phi(r,\theta)=\system{ll}{\ln(r)+C_0 & \textnormal{if}\ n=2\\ r^\alpha \psi(\theta) & \textnormal{if}\ n\geq 3,}
\end{align*}
where $\psi$ is the first Dirichlet eigenfunction of $\Delta^{\sss^{n-1}}\Big|_{\Lambda}$ with associated first eigenvalue $\lambda_1>0$ and $\alpha\in \rr$ satisfies the identity
\begin{align*}
\alpha(\alpha+n-2)-\lambda_1 =0.
\end{align*}
By the nodal domain theorem, it follows that $\phi>0$ in $\Omega$ and thus $(-\Delta+c)\phi \geq 0$. Moreover, by construction, $\phi$ diverges as $|x|\to +\infty$. By Lemma \ref{Lem:BCNLemma}, the claim follows.
\end{proof}

Using a different point of view, we can interpret Theorem \ref{Thm:EuclideanMaximumPrinciple} in terms of a the Dirichlet-parabolicity of the domain $\Omega$.

\begin{definition}
Given a Riemannian manifold $(M,g)$ without boundary, we say that a domain $\Omega \subseteq M$ is \textnormal{Dirichlet parabolic} ($\mathcal{D}$-\textnormal{parabolic}) if the unique bounded solution $u\in C^0(\overline{\Omega})\cap C^\infty(\Omega)$ to the problem
\begin{align*}
\system{ll}{-\lap u=0 & \inn \Omega\\ u=0 & \onn \partial \Omega}
\end{align*}
is the constant null function.
\end{definition}

\begin{remark}
Note that in the definition of $\dpar$ity the boundary of the manifold (domain) at hand does not necessarily have to be smooth.
\end{remark}

For an interesting work about Dirichlet parabolicity, containing a detailed overview about the topic, we suggest \cite{pessoa2017dirichlet}.

As an application of what done so far, we get that any domain $\Omega\subset \rr^n$ contained in the complement of a cone is $\dpar$. 

\begin{corollary}\label{Cor:EuclideanDPar}
If $\Omega \subset \rr^n$, $n\geq 2$, is a (possibly unbounded) domain whose closure is contained in the complement of a non-degenerate solid cone $\mathcal{C}\subset \rr^n$, then $\Omega$ is $\mathcal{D}$-parabolic.
\end{corollary}

\begin{proof}
Fixed any bounded function $u\in C^0(\overline{\Omega})\cap C^\infty(\Omega)$ satisfying
\begin{align*}
\system{ll}{-\lap u=0 & \inn \Omega \\ u=0 & \onn \partial \Omega,}
\end{align*}
by Theorem \ref{Thm:EuclideanMaximumPrinciple} we get $u\leq 0$. Applying the same argument to $v=-u$, it also follows that $u\geq 0$, obtaining $u\equiv 0$.
\end{proof}

\subsection{From Euclidean space to warped products}
Clearly, previous construction is strongly based on the fact that the Euclidean space is a model manifold. Using this viewpoint, a natural question could be the following
\begin{align*}
\begin{array}{c}
Can\ we\ retrace\ what\ we\ have\ done\ so\ far\ to\ obtain\ a\ suitable\ barrier\ \phi\\ on\ any\ warped\ product\ manifold\ M=I\times_\sigma N?
\end{array}
\end{align*}

\begin{remark}
When we consider $\rr^n$ as a warped product manifold, the cone $\mathcal{C}$ (whose vertex coincides with the pole $o$) can be seen as a strip that extends along the "radial" direction.
\end{remark}

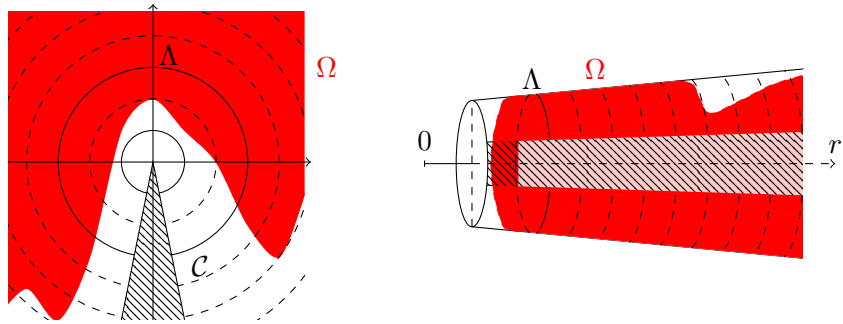
\begin{center}
\begin{tikzpicture}[scale=0.42]
\path[draw, red]  plot [smooth] coordinates{(-16,-5) (-15,-4) (-14,-5) (-13,-3) (-12,1) (-11,2) (-10,1) (-9,0) (-8,-2) (-7,-3) (-6.2,-1)};
\path[fill, red, opacity=0.4]  plot [smooth] coordinates{(-16,-5) (-15,-4) (-14,-5) (-13,-3) (-12,1) (-11,2) (-10,1) (-9,0) (-8,-2) (-7,-3) (-6.2,-1)} -- (-6.2,4.8) -- (-16,4.8) -- (-16,5);
\path[draw] (-10,0) arc(0:360:1 and 1);
\path (-10.5,3.4) node{$\Lambda$};
\path (-9.5, -3.3) node{$\mathcal{C}$};

\path[draw, dashed, opacity=0.5] (-9,0) arc(0:360: 2 and 2);
\path[draw] (-8,0) arc(0:360: 3 and 3);
\path[draw, dashed, opacity=0.5] (-7,0) arc(0:360: 4 and 4);
\path[draw, dashed, opacity=0.5] (-6,0) arc(0:360: 5 and 5);
\path[draw, dashed, opacity=0.5] (-5,0) arc(0:360: 6 and 6);
\fill[white] (-16,4.8)--(-6.2,4.8)--(-6.2,6.6)--(-16,6.6)--cycle;
\fill[white] (-16,-5)--(-6.2,-5)--(-6.2,-6.6)--(-16,-6.6)--cycle;
\fill[white] (-19,4.8)--(-19,-5)--(-15.6,-5)--(-15.6,4.8)--cycle;
\fill[white] (-6.2,4.8)--(-6.2,-4.8)--(-4.4,-4.8)--(-4.4,4.8)--cycle;
\path[draw,->] (-15.6,0)--(-6,0);
\path[red] (-5.5, 3) node{$\Omega$};
\path[fill, white] (-12,-5)--(-11,0)--(-10,-5)--(-12,-5);
\path[draw] (-12,-5)--(-11,0)--(-10,-5);
\path[fill, pattern=north west lines,pattern color=black, opacity=0.5] (-12,-5)--(-11,0)--(-10,-5)--(-11,-5);
\path[draw,->] (-11,-5)--(-11,5);
\end{tikzpicture}
\hfill
\begin{tikzpicture}[scale=0.42]
\path[draw] (0,0) arc(0:360:0.5 and 2);
\path[draw] (-0.5,2)--(10,3);
\path[draw] (-0.5,-2)--(10,-3);
\path[draw, name path=pathc] ({0.5*cos(20)-0.5},{2*sin(20)})--({0.75*cos(20)+10},{3*sin(20)});
\path[draw, name path=pathd] ({0.5*cos(-20)-0.5},{2*sin(-20)})--({0.75*cos(-20)+10},{3*sin(-20)});

\path[fill, red, opacity=0.2] (1,{2*(2+43)*sin(90)/(42)}) --  plot [smooth] coordinates{(1,{2*(2+43)*sin(90)/(42)})  (0.5,2) (0.2,1) (0.1, 0) (0.2,-1) (0.5,-2) (1,-{2*(2+43)*sin(90)/(42)})} -- (1,-{2*(2+43)*sin(90)/(42)}) -- (10,-{2*(20+43)/(42)}) -- (10,2) -- plot [smooth] coordinates{(10,2.8) (9, 2.5) (8,2) (7,1.6) (6.5, 2.5) (6,{2*(12+43)*sin(90)/(42)})} -- cycle;

\path[fill, white] (1,{2*(2+43)*sin(90)/(42)}) -- plot [smooth] coordinates{(1,{2*(2+43)*sin(90)/(42)}) (2,1) (3,1) (4,1) (5,1.5) (6,{2*(12+43)*sin(90)/(42)})} -- (6,{2*(12+43)*sin(90)/(42)}) -- cycle;

\path[fill, red, opacity=0.4] (1,{2*(2+43)*sin(90)/(42)}) -- plot [smooth] coordinates{(1,{2*(2+43)*sin(90)/(42)}) (2,1) (3,1) (4,1) (5,1.5) (6,{2*(12+43)*sin(90)/(42)})} -- (6,{2*(12+43)*sin(90)/(42)}) -- cycle;

\path[fill, white] (1,-{2*(2+43)*sin(90)/(42)}) -- plot [smooth] coordinates{(1,-{2*(2+43)*sin(90)/(42)}) (2, -2) (3,-1) (4,-1) (5,-1) (6,-1.1) (7,-1.2) (8,-1.2) (9,-1.5) (10,-1.7)} -- (10,-{2*(20+43)/(42)})-- cycle;

\path[fill, red, opacity=0.4] (1,-{2*(2+43)*sin(90)/(42)}) -- plot [smooth] coordinates{(1,-{2*(2+43)*sin(90)/(42)}) (2, -2) (3,-1) (4,-1) (5,-1) (6,-1.1) (7,-1.2) (8,-1.2) (9,-1.5) (10,-1.7)} -- (10,-{2*(20+43)/(42)})-- cycle;

\path[name path=patha] (1,0) arc(0:80: {0.5*(2+43)/(42)} and {2*(2+43)/(42)});
\path[name path=pathb] (1,0) arc(0:-80: {0.5*(2+43)/(42)} and {2*(2+43)/(42)});
\path (1.4,2.7) node{$\Lambda$};
\path[red] (3.4,3) node{$\Omega$};

\path[draw] (2,0) arc(0:80: {0.5*(4+43)/(42)} and {2*(4+43)/(42)});
\path[draw] (2,0) arc(0:-80: {0.5*(4+43)/(42)} and {2*(4+43)/(42)});

\path[draw, dashed, opacity=0.5] (3,0) arc(0:80: {0.5*(6+43)/(42)} and {2*(6+43)/(42)});
\path[draw, dashed, opacity=0.5] (3,0) arc(0:-80: {0.5*(6+43)/(42)} and {2*(6+43)/(42)});

\path[draw, dashed, opacity=0.5] (4,0) arc(0:80: {0.5*(8+43)/(42)} and {2*(8+43)/(42)});
\path[draw, dashed, opacity=0.5] (4,0) arc(0:-80: {0.5*(8+43)/(42)} and {2*(8+43)/(42)});

\path[draw, dashed, opacity=0.5] (5,0) arc(0:80: {0.5*(10+43)/(42)} and {2*(10+43)/(42)});
\path[draw, dashed, opacity=0.5] (5,0) arc(0:-80: {0.5*(10+43)/(42)} and {2*(10+43)/(42)});

\path[draw, dashed, opacity=0.5] (6,0) arc(0:80: {0.5*(12+43)/(42)} and {2*(12+43)/(42)});
\path[draw, dashed, opacity=0.5] (6,0) arc(0:-80: {0.5*(12+43)/(42)} and {2*(12+43)/(42)});

\path[draw, dashed, opacity=0.5] (7,0) arc(0:80: {0.5*(14+43)/(42)} and {2*(14+43)/(42)});
\path[draw, dashed, opacity=0.5] (7,0) arc(0:-80: {0.5*(14+43)/(42)} and {2*(14+43)/(42)});

\path[draw, dashed, opacity=0.5] (8,0) arc(0:80: {0.5*(16+43)/(42)} and {2*(16+43)/(42)});
\path[draw, dashed, opacity=0.5] (8,0) arc(0:-80: {0.5*(16+43)/(42)} and {2*(16+43)/(42)});

\path[draw, dashed, opacity=0.5] (9,0) arc(0:80: {0.5*(18+43)/(42)} and {2*(18+43)/(42)});
\path[draw, dashed, opacity=0.5] (9,0) arc(0:-80: {0.5*(18+43)/(42)} and {2*(18+43)/(42)});

\path[draw, dashed, opacity=0.5] (10,0) arc(0:80: {0.5*(20+43)/(42)} and {2*(20+43)/(42)});
\path[draw, dashed, opacity=0.5] (10,0) arc(0:-80: {0.5*(20+43)/(42)} and {2*(20+43)/(42)});

\path[draw, red] plot [smooth] coordinates{(1,{2*(2+43)*sin(90)/(42)}) (2,1) (3,1) (4,1) (5,1.5) (6,{2*(12+43)*sin(90)/(42)})};

\path[draw, dashed, red] plot [smooth] coordinates{(6,{2*(12+43)*sin(90)/(42)}) (6.5, 2.5) (7,1.6) (8,2) (9, 2.5) (10,2.8)};

\path[draw, dashed, red] plot [smooth] coordinates{(1,{2*(2+43)*sin(90)/(42)})  (0.5,2) (0.2,1) (0.1, 0) (0.2,-1) (0.5,-2) (1,-{2*(2+43)*sin(90)/(42)})};

\path[draw, red] plot [smooth] coordinates{(1,-{2*(2+43)*sin(90)/(42)}) (2, -2) (3,-1) (4,-1) (5,-1) (6,-1.1) (7,-1.2) (8,-1.2) (9,-1.5) (10,-1.7)};

\path [name intersections={of = patha and pathc}];
\coordinate (A) at (intersection-1);

\path [name intersections={of = pathb and pathd}];
\coordinate (B) at (intersection-1);

\path[fill, white!80!red] (A) --({0.75*cos(20)+10},{3*sin(20)}) -- ({0.75*cos(-20)+10},{3*sin(-20)}) -- (B)--cycle;

\path[fill, pattern=north west lines,pattern color=black, opacity=0.5] ({0.5*cos(20)-0.5},{2*sin(20)})--({0.5*cos(20)-0.5},{2*sin(20)}) arc(20:-20:0.5 and 2) -- ({0.75*cos(-20)+10},{3*sin(-20)}) -- ({0.75*cos(20)+10},{3*sin(20)}) -- cycle;
\path[fill,white] (10,3)--(10,-3)--(11,-3)--(11,3)--cycle;

\path[draw, dashed] ({2-(4+43)/(42)},0) arc(180:260: {0.5*(4+43)/(42)} and {2*(4+43)/(42)});
\path[draw, dashed] ({2-(4+43)/(42)},0) arc(180:100: {0.5*(4+43)/(42)} and {2*(4+43)/(42)});
\path[draw,dashed,->] (-0.5,0)--(11,0) node[anchor=south]{$r$};
\path[draw] (-2,0)--(-0.5,0);
\path[draw] (-2,-0.1)--(-2,0.1) node[anchor=south]{$0$};
\path[draw,dashed] (-0.5,-2)--(-0.5,2);
\path[draw,white] (-2,-6.5) circle (1pt);
\end{tikzpicture}
\end{center}

If we want to retrace the same construction step by step, we need the existence (and the positiveness) of the first eigenfunction $\phi$ of $\Delta^N\big|_{\Lambda}$. In particular, this means that the manifold $N$ has to be compact. Whence, assuming that $\phi$ takes the form $\phi(r,\xi)=h(r) \psi(\xi)$ with $\psi$ nonnegative first Dirichlet eigenfunction on a fixed subdomain $\Lambda\subset N$, by the structure of the Laplace-Belatrami operator acting on warped product manifolds, the inequality $(-\Delta+c)\phi \geq 0$ reduces to
\begin{align}\label{Eq:RadialLaplacianApp}
\partial_r^2 h + (n-1) \frac{\sigma'}{\sigma} \partial_r h - \left(\frac{\lambda_1}{\sigma^2}+c\right) h \leq 0 
\end{align}
and, in general, it is not easy to prove the existence of a positive solution to \eqref{Eq:RadialLaplacianApp} that satisfies the asymptotic condition $h\xrightarrow[]{r\to +\infty}+\infty$. This means that we are able to generalize Theorem \ref{Thm:EuclideanMaximumPrinciple} only requiring strong assumptions on the manifold at hand.

\subsection{$\mathcal{D}$-parabolicity and maximum principle for unbounded domains of warped product manifolds with compact leaves}\label{Sec:WarpedCompact}
Let $M=\rr_{\geq 0}\times_\sigma N$ be a warped product manifold, with $\sigma:\rr_{\geq 0}\to \rr_{>0}$ a positive smooth function and $N$ a closed manifold. Observe that, up to double $M$, we can equivalently assume $I=\rr$ (and thus that the manifold is complete). In what follows, we consider $\Omega$ an unbounded domain whose closure is contained in the strip $(0,+\infty) \times \Lambda$, where $\Lambda \subset N$ is a non-empty, connected open subset of $N$ (with smooth boundary $\partial \Lambda$) such that $\overline{\Lambda}\neq N$.

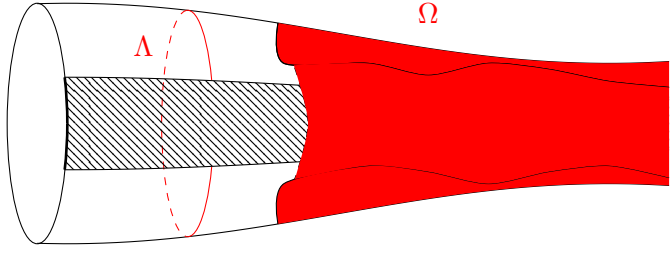
\begin{center}
\begin{tikzpicture}[scale=0.8]
\draw (0,0) arc(0:360:0.5 and 2);
\draw[domain=-0.5:10, smooth, variable=\x] plot (\x, {cos(20*\x)/2 +1.5});
\draw[domain=-0.5:10, smooth, variable=\x] plot (\x, {-cos(20*\x)/2 -1.5});
\draw[red] (2,{-cos(40)/2 - 1.5}) arc(-90:90:{sin(60)/2} and {cos(40)/2 + 1.5});
\draw[dashed,red] (2,{cos(40)/2 + 1.5}) arc(90:270:{sin(60)/2} and {cos(40)/2 + 1.5});
\draw[white, line width=1] ({2+sin(60)/2},0) arc(0:22.5:{sin(60)/2} and {cos(40)/2 + 1.5});
\draw[white, line width=1] ({2+sin(60)/2},0) arc(0:-22.5:{sin(60)/2} and {cos(40)/2 + 1.5});
\draw[line width=1] (0,0) arc(0:26:{sin(60)/2} and {cos(60)/2 + 1.5});
\draw[line width=1] (0,0) arc(0:-26:{sin(60)/2} and {cos(60)/2 + 1.5});
\fill[red] (1.6,1.3) node[anchor=east]{$\Lambda$};
\draw[domain=-0.05:10, smooth, variable=\x] plot (\x, {sin(20)*(cos(20*\x)/2 + cos(60)/2 + 1.5)});
\draw[domain=-0.05:10, smooth, variable=\x] plot (\x, {-sin(20)*(cos(20*\x)/2 + cos(60)/2 + 1.5)});
\path ({0},0) arc(0:22:0.5 and {cos(60)/2 + 1.5}) coordinate (a);
\path ({0},0) arc(0:-22:0.5 and {cos(60)/2 + 1.5}) coordinate (b);
\fill[pattern=north west lines,pattern color=black, opacity=0.5] (a) -- plot [domain=0:10] (\x, {sin(20)*(cos(20*\x)/2 + cos(60)/2 + 1.5)}) -- (10,0) -- (0,0);
\fill[pattern=north west lines,pattern color=black, opacity=0.5] (b) -- plot [domain=0:10] (\x, {-sin(20)*(cos(20*\x)/2 + cos(60)/2 + 1.5)}) -- (10,0) -- (0,0);
\path[name path=patha] plot [smooth] coordinates {(3.5, {cos(70)/2 +1.5}) (3.6, 1) (5,1) (6,0.8) (7,1) (8,0.9) (9,0.7) (10,0.6)};
\path[name path=pathb] plot [smooth] coordinates {(3.5, {-cos(70)/2 -1.5}) (3.6,-1) (5,-0.7) (6,-0.8) (7,-1) (8,-0.8) (9,-0.7) (10,-0.9)};
\path[draw, name path=pathc, dashed] plot [smooth] coordinates {(3.5, {cos(70)/2 +1.5}) (4,0) (3.5, {-cos(70)/2 -1.5})};
\fill[red, opacity=0.5] plot [domain=3.5:10] (\x, {-cos(20*\x)/2 -1.5}) --  plot [smooth] coordinates {(10,-0.9) (9,-0.7) (8,-0.8) (7,-1) (6,-0.8) (5,-0.7) (3.6,-1) (3.5, {-cos(70)/2 -1.5})};
\fill[red, opacity=0.5] plot [domain=3.5:10] (\x, {cos(20*\x)/2 +1.5}) --  plot [smooth] coordinates { (10,0.6) (9,0.7) (8,0.9) (7,1) (6,0.8) (5,1) (3.6, 1) (3.5, {cos(70)/2 +1.5})};
\path [name intersections={of = patha and pathc}];
\coordinate (A) at (intersection-2);
\path [name intersections={of = pathb and pathc}];
\coordinate (B) at (intersection-2);
\path[draw] plot [smooth] coordinates {(3.5, {cos(70)/2 +1.5}) (3.6, 1) (5,1) (6,0.8) (7,1) (8,0.9) (9,0.7) (10,0.6)};
\path[draw] plot [smooth] coordinates {(3.5, {cos(70)/2 +1.5}) (3.6, 1) (5,1) (6,0.8) (7,1) (8,0.9) (9,0.7) (10,0.6)};
\path[draw] plot [smooth] coordinates {(3.5, {-cos(70)/2 -1.5}) (3.6,-1) (5,-0.7) (6,-0.8) (7,-1) (8,-0.8) (9,-0.7) (10,-0.9)};
\fill[red, opacity=0.2] plot [smooth] coordinates {(A) (4,0) (B)} -- plot [smooth] coordinates {(B) (5,-0.7) (6,-0.8) (7,-1) (8,-0.8) (9,-0.7) (10,-0.9)} -- plot [smooth] coordinates { (10,0.6) (9,0.7) (8,0.9) (7,1) (6,0.8) (5,1) (A)};
\path[red] (6,1.5) node[anchor=south]{$\Omega$};
\end{tikzpicture}
\end{center}

While at the beginning of this section we explained how to prove $\dpar$ity using Lemma \ref{Lem:BCNLemma}, for more general warped product manifolds we will apply the following Dirichlet-Khas'minskii test (see \cite[Lemma 14]{pessoa2017dirichlet}) to subdomains of the ambient manifold.

\begin{lemma}[$\mathcal{D}$-Khas'minskii test]
Given a Riemannian manifold $(M,g)$ with boundary $\partial M\neq \emptyset$, if there exists a compact set $K\subset M$ and a function $0\leq \phi\in C^0(M\setminus \inte K)\cap W^{1,2}\loc (\inte M\ \setminus K)$ such that $\phi(x)\to \infty$ as $d^M(x,x_0)\to \infty$ for some (any) $x_0\in M$, and
\begin{align*}
-\int_{\inte M\ \setminus K} g(\nabla\phi, \nabla \rho) & \leq 0 \\ &\forall 0\leq \rho \in C^0(M\setminus \inte K) \cap W^{1,2}\loc (\inte M\ \setminus K),
\end{align*}
then $M$ is $\dpar$.
\end{lemma}

Before stating the main theorem of this section we briefly recall that the radial Ricci curvature $\ric_{rr}$ at a point $p=(r,\xi)$ of a warped product manifold $M=I\times_\sigma N$ is given by
\begin{align*}
\ric_{rr}(p)=\ric \left( \frac{\partial}{\partial r},\frac{\partial}{\partial r} \right)(p)=-\frac{\sigma''(r)}{\sigma(r)}.
\end{align*}
In particular, on noting that $\sigma(r)>0$ for every $r\in I$, we get
\begin{align*}
\ric_{rr}(p) \geq 0\ \ (\textnormal{resp.}\ \leq 0) \spc \Leftrightarrow \spc \sigma''(r) \leq 0\ \ (\textnormal{resp.}\ \geq 0).
\end{align*}

\begin{theorem}\label{Thm:DparabolicSigma}
Let $M=\rr_{\geq 0}\times_\sigma N$ be a warped product manifold of dimension $\textnormal{dim}(M)\geq 2$, where $\sigma:\rr_{\geq 0}\to \rr_{>0}$ is a positive smooth function and $N$ is a closed manifold. Consider $\Omega\subset M$ an unbounded domain whose closure is contained in the strip $[0,+\infty)\times \Lambda$, where $\Lambda \subset N$ is a non-empty, smooth and connected open subset of $N$ such that $\overline{\Lambda}\neq N$. Assume that either one of the following conditions is satisfied
\begin{enumerate}
\item $\ric_{rr}\leq 0$ eventually and $\exists \lim_{r\to \infty} \sigma(r)=c\in [0,+\infty)$;
\item $\ric_{rr}\geq 0$ eventually and $\exists \lim_{r\to \infty} \sigma(r)=c\in (0,+\infty]$;
\item $\sigma\in O(r^\beta)$ for $0<\beta<\frac{1}{2}$ as $r\to +\infty$ and $\frac{\sigma'}{\sigma}\in L^\infty$ eventually.
\end{enumerate}
Then $\overline{\Omega}$ is $\dpar$.
\end{theorem}

\begin{proof}
We recall that $\Omega$ is $\dpar$ if every $u\in C^\infty(\Omega)\cap C^0(\overline{\Omega})\cap L^\infty(\Omega)$ satisfying the Dirichlet problem
\begin{align}\label{Eq:DirichletProblem}
\system{ll}{-\Delta u=0 & \inn \Omega \\ u=0 & \onn \partial \Omega}
\end{align}
vanishes everywhere. By the invariance of $\dpar$ity by removing compact domains, it is enough to prove that there exists an appropriate compact subset $K\subset \Omega$ such that the resulting subdomain $U:=\Omega\setminus K$ is $\dpar$. To this end, in turn, following the philosophy of Khas'minskii test, we only have to find a nonnegative function $\phi \in C^0(\overline{U})\cap W^{1,2}_{\loc}(U)$ satisfying the conditions
\begin{align*}
\system{l}{-\Delta \phi \geq 0 \\ \underset{\underset{x\in \Omega}{d^M(p_0,x)\to \infty}}{\lim} \phi(x) = +\infty}
\end{align*}
for any fixed $p_0 \in M$. Indeed, in this case given any solution $u\in C^\infty (U)\cap C^0(\overline{U})\cap L^\infty(U)$ of \eqref{Eq:DirichletProblem}, suppose by contradiction that $\sup_U u >0$. Then there exists $x_0, x_1 \in U$ such that $\sup_U u \geq u(x_1)> u(x_0)=:u_0>0$. Define $v:=u-u_0-\e \phi$, for $\e$ small enough so that  $v(x_1)>0$, and set $W:=\bra{x\in U\ :\ v(x)>0}$. Then $x_1 \in W$ and $W$ is bounded since $\phi \to +\infty$ as $d^M(p_0,x)\to \infty$. By the fact that $\Delta v \geq 0$ weakly in $W$ and $v\leq 0$ on $\partial W$, using the strong maximum principle we get $v\leq 0$ on $W$, thus obtaining a contradiction. It follows that $u\leq 0$. By applying the same argument to the function $-u$, we conclude $u\equiv 0$, as desired.

It remains to prove the existence of the function $\phi$ and the corresponding compact set $K$. Thanks to the structure of the warped product manifold, we can assume $\phi$ to be of the form $\phi(r,\xi)=h(r)\psi(\xi)$. So, let $\psi$ be the positive first Dirichlet eigenfunction of the Laplacian on $\Lambda$
\begin{align*}
\system{ll}{
-\Delta_\Lambda \psi =\lambda_1 \psi \geq 0 & \inn \Lambda \\ \psi=0 & \onn \partial \Lambda.}
\end{align*}
With this choice the differential inequality $-\lap \phi \geq 0$ is equivalent to the second order ODE
\begin{align}\label{Thm:DparabolicSigma-Eq:hEquation}
h''+(m-1)\frac{\sigma'}{\sigma} h'-\frac{1}{\sigma^2} \lambda_1 h\leq 0.
\end{align}
Whence, we are reduced to find a solution $h$ to \eqref{Thm:DparabolicSigma-Eq:hEquation}. This is obtained via a case by case analysis:
\begin{enumerate}
\item \underline{$\sigma''\geq 0$ eventually and $\exists \lim_{r\to \infty} \sigma(r)=c\in [0,+\infty)$}: by assumption, there exists $A\geq 1$ so that
\begin{align*}
\sigma'' \geq 0 \spc \spc \textnormal{and thus} \spc \spc \sigma \geq c
\end{align*}
in $[A,+\infty)$. This implies that $\sigma' \xrightarrow[]{r\to +\infty} C\leq 0$ and $\sigma'\leq 0$ eventually, so we can assume that $\sigma'\leq 0$ for $r\geq A$. In particular, $-K\leq \sigma' \leq 0$ for a positive constant $K$.\\
Let $h(r):=r$, defined in $[A,+\infty)$: since $h'=1\geq 0$, $h''=0$ and $\sigma'\leq 0$, we get
\begin{align*}
h''+(m-1)\frac{\sigma'}{\sigma} h'-\frac{1}{\sigma^2} \lambda_1 h \leq 0.
\end{align*}
By construction, $h(r)\xrightarrow{r\to +\infty}+\infty$ and $h(r)>0$ in $[A,+\infty)$. Whence, defining $U:=\Omega \cap \left([A,+\infty)\times N \right)$ and taking $\phi(r,\xi)=h(r)\psi(\xi)$, by the previous argument we obtain that $U$ is $\dpar$.

\item[2.a.] \underline{$\sigma''\leq 0$ eventually and $\exists \lim_{r\to \infty} \sigma(r)=c\in (0,+\infty)$}: as in previous case, there exists $A\geq 1$ so that
\begin{align*}
\sigma'' \leq 0 \spc \spc \textnormal{and thus} \spc \spc \sigma \leq c
\end{align*}
in $[A,+\infty)$, implying (w.l.o.g.) $0\leq \sigma'\leq K<+\infty$ in $[A,+\infty)$. Let $\beta\in (0,1)$ and $h(r):=r^\beta$: we get
\begin{align*}
h''+(m-1)\frac{\sigma'}{\sigma} h'-\frac{1}{\sigma^2} \lambda_1 h & \leq (m-1) \frac{\sigma'}{\sigma} \beta r^{\beta-1} - \frac{1}{\sigma^2} \lambda_1 r^\beta\\
& \leq \frac{r^\beta}{\sigma} \left[(m-1) K \beta - \frac{1}{c} \lambda_1 \right]
\end{align*}
and choosing $\beta\in (0,1)$ so that $\left[(m-1) K \beta - \frac{1}{c} \lambda_1\right]\leq 0$, we obtain
\begin{align*}
h''+(m-1)\frac{\sigma'}{\sigma} h'-\frac{1}{\sigma^2} \lambda_1 h\leq 0.
\end{align*}
Since $h$ is positive and diverges as $r\to +\infty$, we can proceed exactly as in previous case, obtaining that $U:=\Omega \cap \left([A,+\infty)\times N \right)$ is $\dpar$.

\item[2.b.] \underline{$\sigma''\leq 0$ eventually and $\exists \lim_{r\to+\infty} \sigma(r)=+\infty$}: by assumption, there exists $A>1$ so that $\sigma''\leq 0 \ \inn [A,+\infty)$. Together with the fact that $\sigma\to +\infty$ as $r\to +\infty$, this implies that $\sigma'$ is decreasing and eventually positive. In particular, $\sigma'\leq K$ is bounded in $[A,+\infty)$. Choosing $h(r)=\sigma^\beta(r)$ for $\beta>0$, we get
\begin{align*}
h''+ & (m-1)\frac{\sigma'}{\sigma} h' -\frac{1}{\sigma^2} \lambda_1 h \\ &=\sigma^{\beta-2} \left[ (\sigma')^2 \beta (\beta+m-2) - \lambda_1 \right] + \underbrace{\beta \sigma^{\beta-1} \sigma''}_{\leq 0}
\end{align*}
in $[A,+\infty)$ and, thanks to the boundedness of $\sigma'$, we can take a positive $\beta$ small enough so that
\begin{align*}
(\sigma')^2 \beta (\beta+m-2) - \lambda_1 \leq 0,
\end{align*}
obtaining
\begin{align*}
h''+(m-1)\frac{\sigma'}{\sigma} h'-\frac{1}{\sigma^2} \lambda_1 h \leq 0
\end{align*}
in $[A,+\infty)$. As in first case, it follows that the subdomain $U:=\Omega \cap \left([A,+\infty)\times N \right)$ is $\dpar$.

\item[3.] \underline{$\sigma\in O(r^\beta)$ for $0<\beta<\frac{1}{2}$ as $r\to \infty$ and $\frac{\sigma'}{\sigma}\in L^\infty$ eventually}: let $K>0$ and $A_0>0$ so that $\frac{\sigma'}{\sigma}<K$ in $[A_0,+\infty)$. Then, under the current assumptions, the function $h(r):=r$ satisfies
\begin{align*}
h''+ & (m-1)\frac{\sigma'}{\sigma} h'- \frac{1}{\sigma^2} \lambda_1 h \\ &< (m-1) K - \frac{1}{\sigma^2} \lambda_1 r\xrightarrow{r\to +\infty} -\infty
\end{align*}
implying that there exists $A>A_0$ so that equation \eqref{Thm:DparabolicSigma-Eq:hEquation} is satisfied in $[A,+\infty)$. Again, it follows that the domain $U:=\Omega \cap \left([A,+\infty)\times N \right)$ is $\dpar$.
\end{enumerate}
As a consequence of the above analysis, we get a $\dpar$ subdomain of the form $U:=\Omega \cap \left([A,+\infty)\times N \right)$, for $A>0$ big enough.
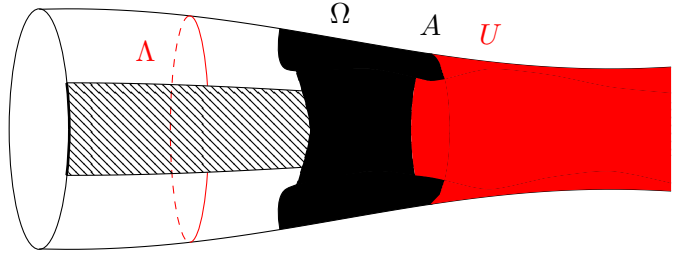
\begin{center}
\begin{tikzpicture}[scale=0.8]
\draw (0,0) arc(0:360:0.5 and 2);
\draw[domain=-0.5:10, smooth, variable=\x] plot (\x, {cos(20*\x)/2 +1.5});
\draw[domain=-0.5:10, smooth, variable=\x] plot (\x, {-cos(20*\x)/2 -1.5});
\draw[red] (2,{-cos(40)/2 - 1.5}) arc(-90:90:{sin(40)/2} and {cos(40)/2 + 1.5});
\draw[dashed,red] (2,{cos(40)/2 + 1.5}) arc(90:270:{sin(40)/2} and {cos(40)/2 + 1.5});
\draw[white, line width=1] ({2+sin(40)/2},0) arc(0:22.5:{sin(40)/2} and {cos(40)/2 + 1.5});
\draw[white, line width=1] ({2+sin(40)/2},0) arc(0:-22.5:{sin(40)/2} and {cos(40)/2 + 1.5});
\draw[line width=1] (0,0) arc(0:26:{sin(60)/2} and {cos(60)/2 + 1.5});
\draw[line width=1] (0,0) arc(0:-26:{sin(60)/2} and {cos(60)/2 + 1.5});
\fill[red] (1.6,1.3) node[anchor=east]{$\Lambda$};
\draw[domain=-0.05:10, smooth, variable=\x] plot (\x, {sin(20)*(cos(20*\x)/2 + cos(60)/2 + 1.5)});
\draw[domain=-0.05:10, smooth, variable=\x] plot (\x, {-sin(20)*(cos(20*\x)/2 + cos(60)/2 + 1.5)});
\path ({0},0) arc(0:22:0.5 and {cos(60)/2 + 1.5}) coordinate (a);
\path ({0},0) arc(0:-22:0.5 and {cos(60)/2 + 1.5}) coordinate (b);
\fill[pattern=north west lines,pattern color=black, opacity=0.5] (a) -- plot [domain=0:10] (\x, {sin(20)*(cos(20*\x)/2 + cos(60)/2 + 1.5)}) -- (10,0) -- (0,0);
\fill[pattern=north west lines,pattern color=black, opacity=0.5] (b) -- plot [domain=0:10] (\x, {-sin(20)*(cos(20*\x)/2 + cos(60)/2 + 1.5)}) -- (10,0) -- (0,0);
\path[name path=patha] plot [smooth] coordinates {(3.5, {cos(70)/2 +1.5}) (3.6, 1) (5,1) (6,0.8) (7,1) (8,0.9) (9,0.7) (10,0.6)};
\path[name path=pathb] plot [smooth] coordinates {(3.5, {-cos(70)/2 -1.5}) (3.6,-1) (5,-0.7) (6,-0.8) (7,-1) (8,-0.8) (9,-0.7) (10,-0.9)};
\path[draw, name path=pathc, dashed] plot [smooth] coordinates {(3.5, {cos(70)/2 +1.5}) (4,0) (3.5, {-cos(70)/2 -1.5})};
\path [name intersections={of = patha and pathc}];
\coordinate (A) at (intersection-2);
\path [name intersections={of = pathb and pathc}];
\coordinate (B) at (intersection-2);

\path[black] (4.5,1.6) node[anchor=south]{$\Omega$};
\path[draw, name path=pathd] (6, {cos(120)/2 +1.5}) arc(90:-90:{sin(140)/2} and {cos(120)/2 + 1.5});
\path[draw,dashed, name path=pathe] (6, {cos(120)/2 +1.5}) arc(90:270:{sin(140)/2} and {cos(120)/2 + 1.5});
\path [name intersections={of = pathd and patha}];
\coordinate (C) at (intersection-1);
\path [name intersections={of = pathd and pathb}];
\coordinate (D) at (intersection-1);
\path [name intersections={of = pathe and patha}];
\coordinate (E) at (intersection-1);
\path [name intersections={of = pathe and pathb}];
\coordinate (F) at (intersection-1);

\path[draw] plot [smooth] coordinates {(3.5, {cos(70)/2 +1.5}) (3.6, 1) (5,1) (6,0.8) (C) (7,1) (8,0.9) (9,0.7) (10,0.6)};
\path[draw] plot [smooth] coordinates {(3.5, {-cos(70)/2 -1.5}) (3.6,-1) (5,-0.7) (6,-0.8) (D) (7,-1) (8,-0.8) (9,-0.7) (10,-0.9)};
\fill[black, opacity=0.3] plot [domain=3.5:6] (\x, {cos(20*\x)/2 +1.5}) -- (6, {cos(120)/2 +1.5}) arc(90:60:{sin(140)/2} and {cos(120)/2 + 1.5}) -- plot [smooth] coordinates { (C) (6,0.8) (5,1) (3.6, 1) (3.5, {cos(70)/2 +1.5})};
\fill[black, opacity=0.3] plot [domain=3.5:6] (\x, {-cos(20*\x)/2 -1.5}) -- (6, {-cos(120)/2 -1.5}) arc(-90:-60:{sin(140)/2} and {cos(120)/2 + 1.5}) -- plot [smooth] coordinates { (D) (6,-0.8) (5,-0.7) (3.6,-1) (3.5, {-cos(70)/2 -1.5})};
\path (6, {cos(120)/2 +1.5}) arc(90:60:{sin(140)/2} and {cos(120)/2 + 1.5}) coordinate (Z1);
\path (6, {-cos(120)/2 -1.5}) arc(-90:-60:{sin(140)/2} and {cos(120)/2 + 1.5}) coordinate (Z2);
\path (6, {cos(120)/2 +1.5}) arc(90:137:{sin(140)/2} and {cos(120)/2 + 1.5}) coordinate (Z3);
\path (6, {-cos(120)/2 -1.5}) arc(-90:-142:{sin(140)/2} and {cos(120)/2 + 1.5}) coordinate (Z4);

\fill[red, opacity=0.5] plot [domain=6:10] (\x, {-cos(20*\x)/2 -1.5}) --  plot [smooth] coordinates {(10,-0.9) (9,-0.7) (8,-0.8) (7,-1) (D)} -- (Z2) arc(-90:-60:{sin(140)/2} and {cos(120)/2 + 1.5});
\fill[red, opacity=0.5] plot [domain=6:10] (\x, {cos(20*\x)/2 +1.5}) --  plot [smooth] coordinates { (10,0.6) (9,0.7) (8,0.9) (7,1) (C)} -- (Z1) arc(60:90:{sin(140)/2} and {cos(120)/2 + 1.5});
\fill[black, opacity=0.1] plot [smooth] coordinates {(A) (4,0) (B)} -- plot [smooth] coordinates {(B) (5,-0.7) (Z4)} -- (Z4) arc(-142:-223:{sin(140)/2} and {cos(120)/2 + 1.5}) -- plot [smooth] coordinates {(Z3) (5,1) (A)};
\fill[red, opacity=0.2] (Z1) arc(60:-60:{sin(140)/2} and {cos(120)/2 + 1.5})  -- plot [smooth] coordinates {(D) (7,-1) (8,-0.8) (9,-0.7) (10,-0.9)} -- plot [smooth] coordinates { (10,0.6) (9,0.7) (8,0.9) (7,1) (C)};
\fill[red, opacity=0.2] plot [smooth] coordinates {(Z4) (6,-0.8) (D)} -- (Z2) arc(-60:60:{sin(140)/2} and {cos(120)/2 + 1.5}) -- plot [smooth] coordinates {(C) (6,0.8) (Z3)} -- (Z3) arc(137:218:{sin(140)/2} and {cos(120)/2 + 1.5});
\path[red] (7,1.25) node[anchor=south]{$U$};
\path (6,1.4) node[anchor=south]{$A$};
\end{tikzpicture}
\end{center}
Since $\Omega\setminus U = \left([0,A]\times N\right) \cap \Omega$ is compact in $\Omega$ and $U$ is $\dpar$, by \cite[Corollary 11]{pessoa2017dirichlet} the domain $\Omega$ is itself $\dpar$, thus completing the proof.
\end{proof}

A direct application of Theorem \ref{Thm:DparabolicSigma} gives the following maximum principle for unbounded domains. Its proof is based on a characterization of the $\dpar$ity contained in \cite[Proposition 10]{pessoa2017dirichlet}, which asserts that a Riemannian manifold $X$ with nonempty boundary $\partial X\neq \emptyset$ is $\dpar$ if and only if every subharmonic bounded function $u\in C^0(X)\cap W^{1,2}\loc(\inte X)$ satisfies $\sup_X u=\sup_{\partial X} u$.

\begin{corollary}[Unbounded maximum principle]\label{Cor:Maximum principle cone}
Let $M=\rr_{\geq 0}\times_\sigma N$ be a warped product manifold of dimension $\textnormal{dim}(M)\geq 2$, where $\sigma:\rr_{\geq 0}\to \rr_{>0}$ is a positive smooth function and $N$ a closed manifold. Consider $\Omega\subset M$ an unbounded domain whose closure is contained in the strip $[0,+\infty)\times \Lambda$, where $\Lambda \subset N$ is a non-empty, smooth and connected open subset of $N$ such that $\overline{\Lambda}\neq N$. Moreover, suppose the validity of either one of the following conditions
\begin{enumerate}
\item[a.] $\ric_{rr}\leq 0$ eventually and $\exists \lim_{r\to \infty} \sigma(r)=c\in [0,+\infty)$;
\item[b.] $\ric_{rr}\geq 0$ eventually and $\exists \lim_{r\to \infty} \sigma(r)=c\in (0,+\infty]$;
\item[c.] $\sigma\in O(r^\beta)$ for $0<\beta<\frac{1}{2}$ as $r\to \infty$ and $\frac{\sigma'}{\sigma}\in L^\infty$ eventually.
\end{enumerate}
If $u\in C^0(\overline{\Omega})\cap W^{1,2}\loc (\Omega)$ is a bounded above distributional solution of the problem
\begin{align*}
\system{rl}{-\lap u + c\ u \leq 0 & \inn \Omega \\ u\leq 0 & \onn \partial \Omega,}
\end{align*}
where $0\leq c\in C^0(\Omega)$, then
\begin{align*}
u\leq 0 \spc \inn \Omega.
\end{align*}
\end{corollary}
\begin{proof}
Consider $u\in C^0(\overline{\Omega})\cap W^{1,2}\loc (\Omega)$ a bounded above distributional solution to the problem
\begin{align*}
\system{rl}{-\lap u + c\ u \leq 0 & \inn \Omega \\ u\leq 0 & \onn \partial \Omega.}
\end{align*}
If $u^+:=\max\{u,0\}$, by Kato's inequality (see \cite[Proposition A.1]{pigola2021p}) we get
\begin{align*}
\system{ll}{-\Delta u^+ \leq -c u^+ \leq 0 & \inn \Omega \\ u^+ = 0 & \onn \partial \Omega.}
\end{align*}
Using Theorem \ref{Thm:DparabolicSigma} and \cite[Proposition 10]{pessoa2017dirichlet} it follows that $u^+=0$ in $\Omega$, implying $u\leq 0$ in $\Omega$.
\end{proof}

\section{A maximum principle for general unbounded domains in complete manifolds}\label{Sec:ABP and Maximum Principle}

In the present section we aim to prove a Maximum Principle for second order elliptic operators acting on unbounded domains of more general Riemannian manifolds. We stress that in the main theorem of this section, i.e. Theorem \ref{Thm:UnboundedMaximumPrinciple}, we only require the positivity (in the spectral sense) of the operator, with no further assumptions neither on the geometry or on the structure of the ambient manifold.

The result is obtained readapting the work made in the Euclidean case by Samuel Nordmann, \cite{nordmann2021maximum}. Most of the effort consists into recover in a Riemannian setting some classical Euclidean tools. In particular, it will be crucial the achievement of an Alexandroff-Bakelman-Pucci estimate, which will allow us to construct a (generalized) first eigenfunction in unbounded domains. The Maximum Principle will be a straightforward consequence of the existence of such eigenfunction.

\subsection{ABP inequality}

In the very interesting article \cite{cabre1997nondivergent}, Cabré proved a Riemannian version of the Alexandroff-Bakelman-Pucci estimate for elliptic operators in nondivergent form acting on manifolds with nonnegative sectional curvature. In his work, he used the assumption on the sectional curvature to ensure two fundamental tools: the (global) volume doubling property for the Riemannian measure $\dvol$ and the classical Hessian comparison principle by Rauch. In particular, since these two tools (with different curvature bounds) are available in every relatively compact domain $\Omega\subset M$ regardless of any assumption on the sectional curvature of $M$, it is reasonable to expect that we can locally recover the results by Cabré up to multiply by appropriate constants depending on $\Omega$ and on the lower bound of its sectional curvature.

Among its various applications, the ABP inequality is one of the main ingredients used by Berestycki, Nirenberg and Varadhan in \cite{MR1258192} to prove the existence of the \textit{generalized principal eigenfunction} of a second order differential operator $\operl$ on Euclidean domains, that is, a generalization of the notion of eigenfunction to operators acting on possibly nonsmooth or unbounded domains. In this paper we will see how to transplant the construction of the generalized principal eigenfunction into general bounded (and into smooth unbounded) Riemannian domains: this will allow us to prove a maximum principle for uniformly elliptic second order differential operators acting in smooth unbounded domains.

Following the proof in \cite{cabre1997nondivergent}, we get a version of the ABP inequality for uniformly elliptic operators of the form
\begin{align}\label{Eq:Operl}
\mathcal{L} u(x):=\operm u(x) + c(x) u(x),
\end{align}
with
\begin{align*}
\operm u(x) := \dive{A(x)\cdot \nabla u(x)}+g(B(x),\nabla u(x)),
\end{align*}
acting on a bounded Riemannian domain $\Omega\subset M$, where $c\in C^0(M)$ is a continuous function, $B\in C^\infty(M;TM)$ is a smooth vector field and $A\in \textnormal{End}(TM)$ is a positive definite symmetric endomorphism of the tangent bundle $TM$ so that
\begin{align*}
c_0\ g(\xi,\xi) \leq g(A(x)\cdot \xi,\xi)\leq C_0\ g(\xi,\xi) \spc \spc \forall x\in M, \forall \xi \in T_x M
\end{align*}
and
\begin{align*}
g(B(x),B(x))\leq b,\spc |c(x)|\leq b \spc \spc \forall x \in M
\end{align*}
for some positive constants $c_0,C_0$ and $b$. Moreover, we assume that the local coefficients $a_i^j$ of the endomorphism $A$ satisfy
\begin{align}\label{Cond:CoeffA}
\norm{a_i^j}{C^1}\leq a\quad \quad \forall i,j,
\end{align}
where $a\in \rr_{>0}$.
\smallskip

The strategy we adopt to achieve the ABP inequality is strongly based on the existence of a suitable atlas composed by harmonic charts. To this aim, let's start by introducing the following definition.

\begin{definition}\label{Def:Harmonic}
Given an $n$-dimensional Riemannian manifold $(M,g)$, we recall that the $C^1$-\textnormal{harmonic radius of $M$ at $x\in M$}, denoted with $r_h(x)$, is the supremum among all $R>0$ so that there exists a coordinate chart $\phi:B_R(x)\to \rr^n$ with the following properties
\begin{enumerate}[label=(\roman*)]
\item \label{Cond1} $2^{-1} g^{\rr^n}\leq g \leq 2 g^{\rr^n}$ in the local chart $(B_R(x),\phi)$;
\item \label{Cond2} $||\partial_k g_{ij}||_{C^0(B_R(x))}\leq \frac{1}{R}$ for every $k=1,...,n$;
\item $\phi$ is an harmonic map.
\end{enumerate}
\end{definition}
\noindent Defining $r_h(M):=\inf_{x\in M} r_h(x)$, if we suppose that
\begin{align}\label{AssHarm}
|\textnormal{Ric}|\leq K \quad \andd \quad \textnormal{inj}_{(M,g)}\geq i
\end{align}
for some constants $K,i\in \rr_{>0}$, by \cite[Corollary]{hebey1997harmonic} it follows that there exists a constant $r_0=r_0(n,K,i)>0$ so that
\begin{align*}
r_h(M)\geq r_0.
\end{align*}
As a consequence, under the assumptions \eqref{AssHarm} we can choose a cover of harmonic charts (with fixed positive radius) providing a uniform $C^1$-control on the metric and on its derivatives.

\begin{theorem}\label{Thm:ABP}
Let $(M,g)$ be a complete Riemannian manifold of dimension $\textnormal{dim}(M)=n$ and $\Omega\Subset M$ a bounded smooth domain. Denote $\Omega_r:=\left\{ x\in M\ :\ d(x,\Omega)<r\right\}$ for $r>0$.

Then, there exists a positive constant $C=C(n,a,b,c_0,C_0,r_h(\overline{\Omega}),|\Omega|,|\Omega_{r_h(\overline{\Omega})}|)$ such that for every $u\in C^2(\Omega)$ satisfying
\begin{align*}
\system{l}{\mathcal{M}u\geq f\ \inn \Omega \\ \limsup_{x\to \partial \Omega} u(x) \leq 0,}
\end{align*}
it holds
\begin{align}\label{Eq:ABP1}
\sup_\Omega u\leq C\ \textnormal{diam}(\Omega)\norm{f}{L^n(\Omega)}.
\end{align}
\end{theorem}

The key result that we need to prove Theorem \ref{Thm:ABP} is the following Euclidean integral Harnack inequality, whose proof can be found in \cite[Theorem 9.22]{gilbarg1977elliptic}

\begin{theorem}\label{Thm:WeakHarnackTrudinger}
Let $\operl:= a^{ij} \partial_{i}\partial_{j}+b^i \partial_i+c$ be an uniformly elliptic differential operator acting on a bounded domain $U\subset \rr^n$ with
\begin{align*}
c_0\leq [a^{ij}]\leq C_0 \quad and \quad |b^i\partial_i|,|c|\leq b,
\end{align*}
for some positive constants $c_0,C_0$ and $b$, and let $f\in L^n(U)$. If $u\in W^{2,n}(U)$ satisfies $\mathcal{L}u\leq f$ and is nonnegative in a ball $B_{2R}(z)\subset U$, then
\begin{align*}
\left( \fint_{B_R(z)} u^p \right)^{\frac{1}{p}}\leq C_1 \left(\inf_{B_R(z)} u + R\ \norm{f}{L^n(B_{2R}(z))} \right)
\end{align*}
where $p$ and $C_1$ are positive constants depending on $n,\ bR,\ c_0$ and $C_0$.
\end{theorem}

\begin{remark}\label{Rmk:b=0}
If $b=0$, i.e. if $B=b^i\partial_i$ is the null vector field and $c\equiv 0$, then the constants $p$ and $C_1$ in previous theorem do not depend on the radius $R$.
\end{remark}

\begin{remark}\label{Rmk:Extension}
If $\Omega$ is a bounded smooth domain and $u\in C^2(\Omega)\cap C^1(\overline{\Omega})$ satisfies
\begin{align*}
\system{ll}{\mathcal{M}u\leq f & \inn \Omega\\ u\equiv C & \onn \partial \Omega \\ \frac{\partial u}{\partial A\cdot\nu} \leq 0 & \onn \partial \Omega,}
\end{align*}
where $\nu$ is the outward pointing unit vector field normal to $\partial \Omega$, then we can consider a larger bounded smooth domain $\Lambda\Supset \Omega$ and we can extend $u$ and $f$ to $\Lambda$ by imposing $u\equiv C$ and $f\equiv 0$ in $\Lambda \setminus \overline{\Omega}$. In this way we get a function $u\in C^0(\Lambda)\cap W^{2,n}(\Lambda)$ satisfying $\operm u\leq f$ weakly in $\Lambda$, i.e. so that
\begin{align*}
\int_{\Lambda} \left[-g(A\cdot \nabla u,\nabla \phi) + g(B, \nabla u) \phi \right] \dvol \leq \int_\Lambda f \phi \dvol \quad \quad \forall 0\leq \phi \in C^\infty_c(\Lambda).
\end{align*}
\end{remark}

\begin{remark}\label{Rmk:Hopf}
We stress that if $\Omega$ is a bounded smooth domain, $u\in C^2(\Omega)\cap C^1(\overline{\Omega})$ satisfies
\begin{align*}
\system{ll}{\mathcal{M}u\leq 0 & \inn \Omega\\ u\equiv C & \onn \partial \Omega}
\end{align*}
and $x_0\in \partial \Omega$ is a global minimum for $u$ in $\overline{\Omega}$, then
\begin{align*}
\frac{\partial u}{\partial A\cdot\nu}(x_0)\leq 0.
\end{align*}
Indeed, by decomposing $A\cdot \nu = (A\cdot \nu)^\top +(A\cdot \nu)^\bot$, where $(A\cdot \nu)^\top$ and $(A\cdot \nu)^\bot$ are tangential and normal to $\partial \Omega$ respectively, one can check that
\begin{align*}
\frac{\partial u}{\partial A\cdot\nu}(x_0)=(A(x_0)\cdot \nu(x_0))^\bot \frac{\partial u}{\partial \nu}(x_0)=\underbrace{g\Big(A(x_0)\cdot \nu(x_0), \nu(x_0)\Big)}_{>0} \frac{\partial u}{\partial \nu}(x_0)
\end{align*}
where the first equality follows from the fact that $x_0\in \partial \Omega$ is a minimum for $u|_{\partial \Omega}$, implying that the tangential component (to $\partial \Omega$) of $\nabla u$ vanishes at $x_0$. Hence $\frac{\partial u}{\partial A\cdot\nu}(x_0)$ and $\frac{\partial u}{\partial \nu}(x_0)$ have the same sign. By standard Hopf's Lemma it follows that $\frac{\partial u}{\partial A\cdot\nu}(x_0)\leq 0$.
\end{remark}

\begin{remark}
Using the local expression of the differential operator $\mathcal{M}$, we can estimate the constant of Theorem \ref{Thm:WeakHarnackTrudinger} in every local chart in terms of the coefficients $A, B$ and $c$ and of the fist order derivatives of the metric, i.e. in terms of the harmonic radius of $M$ thanks to condition \ref{Cond2}. Indeed, if $X$ is a vector field, in local coordinates
\begin{align*}
\dive{X}= \frac{\partial X^k}{\partial x^k}+ X^t \Gamma^k_{kt}
\end{align*}
obtaining
\begin{align*}
    \textnormal{div}(A\cdot\nabla u)&= \dive{a_i^j \frac{\partial}{\partial x^j}\otimes dx^i \left[g^{hk} \frac{\partial u}{\partial x^k} \frac{\partial}{\partial x^h}\right]} \\
    &=\frac{\partial}{\partial x^j} \left(a_i^j g^{hi} \frac{\partial u}{\partial x^h} \right)+a_i^t g^{hi} \frac{\partial u}{\partial x^h} \Gamma^k_{kt}.
\end{align*}
Hence the differential operator $\mathcal{M}$ writes as
\begin{align*}
\mathcal{M}u&=\dive{A\cdot \nabla u}+ g(B,\nabla u)\\
&=\dive{a_i^j \frac{\partial}{\partial x^j}\otimes dx^i \left[g^{hk} \frac{\partial u}{\partial x^k} \frac{\partial}{\partial x^h}\right]}+g\left(B^j \frac{\partial}{\partial x^j},g^{hk} \frac{\partial u}{\partial x^k} \frac{\partial}{\partial x^h}\right)\\
&= \frac{\partial}{\partial x^j} \left(a_i^j g^{hi} \frac{\partial u}{\partial x^h} \right)+a_i^t g^{hi} \frac{\partial u}{\partial x^h} \Gamma^k_{kt} +B^k \frac{\partial u}{\partial x^k}\\
&=a_i^j g^{hi} \frac{\partial^2 u}{\partial x^j\partial x^h}+ \left(\frac{\partial}{\partial x^j}\left(a_i^j g^{ki} \right)+a_i^t g^{ki} \Gamma^h_{ht}+B^k \right)\frac{\partial u}{\partial x^k}.
\end{align*}
As a consequence, under the assumptions \eqref{AssHarm} the coefficients of $\mathcal{M}$ have the same bounds in every harmonic chart of the manifold $M$. In particular, in Theorem \ref{Thm:WeakHarnackTrudinger} we can chose the same constants $p=p(n,r_h(M),a,b,c_0,C_0)$ and $C=C(n,r_h(M),a,b,c_0,C_0)$ for every harmonic chart, avoiding any dependence on the local chart.

Lastly, we stress that if we consider an operator of the form
\begin{align*}
\mathcal{M}(u)= \textnormal{tr}\left(A\cdot \textnormal{Hess}(u) \right)+g(B,\nabla u),
\end{align*}
then the same conclusion holds true without requiring the condition \eqref{Cond:CoeffA}.
\end{remark}

\begin{proof}[Proof of Theorem \ref{Thm:ABP}]
We start by supposing that $u$ and the coefficients of $\mathcal{M}$ are smooth up to the boundary of $\Omega$. Consider the solution $w$ of the problem
\begin{align*}
\system{ll}{\operm w = -F := - (\operm u)^- \leq 0 & \inn \Omega \\ w=0 & \onn \partial \Omega.}
\end{align*}
By assumption, $u\in C^\infty(\overline{\Omega})$ and so $F=(\operm u)^-$ is Lipschitz in $\overline{\Omega}$, implying that $w\in C^{2,\alpha}(\overline{\Omega})$ for  any $\alpha \in (0,1)$. Moreover, by the standard maximum principle, we have $w\geq 0$. Now consider the function $w-u$: by definition
\begin{align*}
\system{ll}{\operm (w-u)\leq 0 & \inn \Omega \\ w-u\geq 0 & \onn \partial \Omega}
\end{align*}
and, again by standard maximum principle,
\begin{align*}
w\geq u \spc \inn \Omega.
\end{align*}
Take $z_0\in \Omega$ so that $S=w(z_0)=\sup_{\Omega} w>0$ and consider the function $v:=S-w\geq 0$. Let $r:=r_h(\overline{\Omega})$ and consider the $r$-neighbourhood $\Omega_{r}$ of $\Omega$
\begin{align*}
\Omega_{r}:=\{x\in M\ :\ d(x,\Omega)<r\}.
\end{align*}
Since $v|_{\partial \Omega}\equiv S$,
 by Remark \ref{Rmk:Hopf}, we can extend $v$ and $F$ to $\Omega_r$ as done in Remark \ref{Rmk:Extension}.

Observe that, without loss of generality, we can suppose $\textnormal{diam}(\Omega)\geq r$. Otherwise, $\Omega$ is contained in an harmonic local chart and the theorem follows by the standard Euclidean ABP inequality.

Consider an open cover $\mathcal{W}$ of $\overline{\Omega}$ given by
\begin{align*}
\mathcal{W}:=\{(W_1:=B_{r/4}(x_1),\phi_1),...,(W_t:=B_{r/4}(x_t),\phi_t)\}
\end{align*}
satisfying the following assumptions
\begin{itemize}
\item $x_i\in \overline{\Omega}$ for every $i=1,...,t$;
\item $d(x_i,x_j)\geq \frac{r}{8}$ for every $i\neq j$;
\item $\mathcal{W}$ is maximal (by inclusion).
\end{itemize}
For a reference see \cite[Lemma 1.1]{He}. Moreover, observe that by construction
\begin{align*}
\bigcup_{i\leq t} W_i\subset \Omega_r.
\end{align*}
Since every chart of $\mathcal{W}$ is an harmonic chart, then
\begin{align*}
|\Omega_r|\geq \left|\cup_{1\leq i \leq t} B_{r/8}(x_i) \right| = \sum_{i \leq t} |B_{r/8}(x_i)| \geq t 2^{-n/2} |\mathbb{B}_{r/8}| 
\end{align*}
implying that
\begin{align}\label{Eq:AppABPI}
t\leq \frac{|\Omega_r|2^{n/2}}{|\mathbb{B}_{r/8}|}
\end{align}
where $\mathbb{B}_s$ denotes the Euclidean ball of radius $s$. Now let $\mathcal{U}$ and $\mathcal{V}$ the dilated covers obtained from $\mathcal{W}$
\begin{align*}
&\mathcal{U}:=\{(U_1:=B_{r}(x_1),\phi_1),...,(U_t:=B_{r}(x_t),\phi_t)\}\\
&\mathcal{V}:=\{(V_1:=B_{r/2}(x_1),\phi_1),...,(V_t:=B_{r/2}(x_t),\phi_t)\}.
\end{align*}
Observe that
\begin{align*}
W_i\cap W_j \neq \emptyset \quad \Rightarrow \quad \exists B_{r/4}(x_{ij})\subseteq V_i \cap V_j
\end{align*}
which implies, by \ref{Cond1} in Definition \ref{Def:Harmonic},
\begin{equation}\label{Eq:DoublingHarm}
\begin{split}
\frac{|V_j|}{|V_i\cap V_j|} &= \frac{|B_{r/2}(x_j)|}{|V_i\cap V_j|} \leq \frac{|B_{r/2}(x_j)|}{|B_{r/4}(x_{ij})|} \\
& \overset{\ref{Cond1}}{\leq} \frac{2^{n/2}|\mathbb{B}_{r/2}|}{2^{-n/2}|\mathbb{B}_{r/4}|} =  \frac{2^n|\mathbb{B}_{r/2}|}{|\mathbb{B}_{r/4}|} \leq 2^n C_{\rr^n}
\end{split}
\end{equation}
whenever $W_i\cap W_j \neq \emptyset$, where $C_{\rr^n}=2^n$ is the Euclidean doubling constant. It follows that if $W_i\cap W_j \neq \emptyset$
\begin{align}\label{Eq:AppABPII}
\fint_{V_i\cap V_j} v^p \leq C_D \fint_{V_j} v^p
\end{align}
where $C_D:=4^{n}$.

In any local chart $U_i$ we can apply Theorem \ref{Thm:WeakHarnackTrudinger}, obtaining
\begin{equation}\label{Eq:AppABPIII}
\begin{split}
\fint_{V_i} v^p \dint{\textnormal{v}} & \leq 2^n \fint_{\mathbb{B}_{r/2}} (v\circ \phi_i)^p \dint{x} \\
&\leq 2^n C_1^p \left[\inf_{\mathbb{B}_{r/2}} v\circ \phi_i^{-1} + \frac{r}{2} \norm{F\circ \phi_i^{-1}}{L^n(\mathbb{B}_{r})} \right]^p\\
& \leq 2^n C_1^p \left[\inf_{V_i} v + \frac{r}{2} \sqrt{2}\norm{F}{L^n(U_i)} \right]^p
\end{split}
\end{equation}
that implies
\begin{equation}\label{Eq:AppABP0}
\begin{split}
\left(\fint_{V_i} v^p \dint{\textnormal{v}} \right)^{1/p} &\leq \underbrace{2^{n/p} C_1}_{=:\widetilde{C}_1} \left[\inf_{V_i} v + \frac{r}{\sqrt{2}}\norm{F}{L^n(U_i)} \right]\\
& \leq \widetilde{C}_1 \left[\inf_{V_i} v + r\norm{F}{L^n(U_i)} \right] \quad \quad \forall i=1,...,t.
\end{split}
\end{equation}
Summing up over $i=1,...,t$, on the left side of \eqref{Eq:AppABPIII} we have
\begin{equation}\label{Eq:AppABP1}
\begin{split}
\sum_{i\leq t} \fint_{V_i} v^p \geq \frac{1}{|\widehat{\Omega}|} \int_{\widehat{\Omega}} v^p=\fint_{\widehat{\Omega}}v^p
\end{split}
\end{equation}
where
\begin{align*}
\widehat{\Omega}:=\bigcup_{1\leq i\leq t} V_i \subseteq \Omega_r.
\end{align*}
Now let $j\in \{1,...,t\}$ be so that
\begin{align*}
\left(\inf_{V_j} v + r\norm{F}{L^n(U_j)} \right)=\max_{i\leq t} \left(\inf_{V_i} v + r\norm{F}{L^n(U_i)} \right).
\end{align*}
and let $\mathcal{S}:=\{W_{i_1},...,W_{i_m}\}\subseteq \mathcal{W}$ be a sequence of coordinate neighbourhoods joining $W_j=:W_{i_1}$ and $z_0\in W_{i_m}$ and such that
\begin{align*}
&W_{i_q}\neq W_{i_s} \spc \forall q\neq s,\\
&W_{i_q}\cap W_{i_{q+1}}\neq \emptyset \spc \forall q=1,...,m-1.
\end{align*}
We get 
\begin{align*}
\inf_{V_j} v=\inf_{V_{i_1}} v & \leq \inf_{V_{i_1}\cap V_{i_2}} v \\
&\overset{\textnormal{by}\ \eqref{Eq:AppABPII}}{\leq} \left(\fint_{V_{i_1}\cap V_{i_2}} v^p \right)^{1/p}\\
&\overset{\textnormal{by}\ \eqref{Eq:AppABP0}}{\leq} C_D \left(\fint_{V_{i_2}} v^p\right)^{1/p}\\
&\leq C_D \widetilde{C}_1 \left( \inf_{V_{i_2}} v + r\norm{F}{L^n(U_{i_2})}\right)\\
&\leq C_D \widetilde{C}_1 \left( \inf_{V_{i_2}} v + r\norm{F}{L^n(\widetilde{\Omega})}\right)
\end{align*}
where
\begin{align*}
\widetilde{\Omega}=\bigcup_{1\leq i\leq t} U_i.
\end{align*}
Iterating
\begin{align*}
\inf_{V_j} v &\leq (C_D \widetilde{C}_1)^m \left( \inf_{V_{i_m}} v + m\ r\norm{F}{L^n(\widetilde{\Omega})}\right)\\
&=(C_D \widetilde{C}_1)^m \left( m\ r\norm{F}{L^n(\widetilde{\Omega})}\right)\\
&\leq (C_D \widetilde{C}_1)^t \left( t\ \textnormal{diam}(\Omega) \norm{F}{L^n(\widetilde{\Omega})}\right)\\
&=C_2\ \textnormal{diam}(\Omega) \norm{F}{L^n(\widetilde{\Omega})}
\end{align*}
where, using \eqref{Eq:AppABPI}, $C_2:=t(C_D \widetilde{C}_1)^t$ can be bounded from above by
\begin{align*}
C_2 \leq \frac{|\Omega_r|2^{n/2}}{|\mathbb{B}_{r/8}|} (C_D \widetilde{C}_1)^{\frac{|\Omega_r|2^{n/2}}{|\mathbb{B}_{r/8}|}}.
\end{align*}
Observe that, without loss of generality, $C_D\widetilde{C}_1 \geq 1$. In this way we obtain
\begin{equation}\label{Eq:AppABP2}
\begin{split}
\sum_{i\leq t} \widetilde{C}_1^p \left( \inf_{V_i} + r \norm{F}{L^n(U_i)}\right)^p & \leq t \widetilde{C}_1^p \left(\inf_{V_j} v + \textnormal{diam}(\Omega) \norm{F}{L^n(\widetilde{\Omega})}\right)^p\\
& \leq \widetilde{C}_2^p \left(\textnormal{diam}(\Omega) \norm{F}{L^n(\widetilde{\Omega})}\right)^p
\end{split}
\end{equation}
where $\widetilde{C}_2:=t^{1/p}\widetilde{C}_1(C_2+1)$. Using \eqref{Eq:AppABP0}, \eqref{Eq:AppABP1} and \eqref{Eq:AppABP2}, it follows
\begin{align*}
\fint_{\widehat{\Omega}} v^p \leq \widetilde{C}_2^p \left(\textnormal{diam}(\Omega) \norm{F}{L^n(\widetilde{\Omega})} \right)^p
\end{align*}
i.e.
\begin{align}\label{Eq:AppABP3}
\left(\fint_{\widehat{\Omega}} v^p \right)^{1/p} \leq \widetilde{C}_2\ \textnormal{diam}(\Omega) \norm{F}{L^n(\widetilde{\Omega})}.
\end{align}
Recalling that $v\equiv S$ in $\widehat{\Omega}\setminus \Omega$, we get
\begin{align*}
\left( \fint_{\widehat{\Omega}} v^p\right)^{1/p} \geq \left(\frac{1}{|\widehat{\Omega}|} \int_{\widehat{\Omega}\setminus \Omega} v^p \right)^{1/p} \geq \left(\frac{|\widehat{\Omega}\setminus\Omega|}{|\widehat{\Omega}|} \right)^{1/p} S=:\theta^{1/p} S
\end{align*}
and, since $|F|\leq |f| \chi_{\Omega}$, by \eqref{Eq:AppABP3}
\begin{align*}
\left( \fint_{\widehat{\Omega}} v^p \right)^{1/p} \leq \widetilde{C}_2\ \textnormal{diam}(\Omega) \norm{F}{L^n(\widetilde{\Omega})}\leq \widetilde{C}_2\ \textnormal{diam}(\Omega) \norm{f}{L^n(\Omega)}.
\end{align*}
Whence
\begin{align}\label{Eq:AppABP4}
\sup_\Omega w =S\leq C\ \textnormal{diam}(\Omega) \norm{f}{L^n(\Omega)}
\end{align}
where $C=\frac{\widetilde{C}_2}{\theta^{1/p}}$. In particular, previous inequality implies
\begin{align*}
\sup_\Omega w\leq C\ \textnormal{diam}(\Omega)\ |\Omega|^{1/n} \norm{f}{L^\infty(\Omega)}.
\end{align*}

For the general case, i.e. removing the smoothness assumption on $u$ and on the coefficients of $\operm$ up to the boundary, we can proceed by an exhaustion of $\Omega$ by smooth, relatively compact subdomains, as done in \cite[Theorem 2.3]{cabre1997nondivergent}. Indeed, let $\{U_\e\}_{\e>0}$ be a family of relatively compact subdomain of $\Omega$ with smooth boundary so that $u\leq \e$ in $\Omega\setminus U_\e$ (recall that $\limsup_{x\to \partial \Omega} u(x)\leq 0$) and satisfying $\bigcup_\e U_\e =\Omega$ and define $u_\e=u-\e\in C^2(\overline{U_\e})$. If we consider the following sequences
\begin{itemize}
\item $\bra{u_{k}}_k\subset C^\infty(\overline{U_\e})$ approximating uniformly $u$ and its derivatives up to order 2;
\item $\bra{A_{k,\e}}_k\subset \textnormal{End}(TM)$ a sequence of positive definite symmetric endomorphisms of the tangent bundle $TM$ whose coefficients are smooth and converge to the ones of $A$ in $W^{1,n}(U_\e)$;
\end{itemize}
then, defining $u_{k,\e}:=u_k-\e$ and $F_{k,\e}:=\bigg(\dive{A_{k,\e}\cdot \nabla u_{k,\e}}+g(B,\nabla u_{k,\e}) \bigg)^-$, by \eqref{Eq:AppABP4} in previous step we get
\begin{align*}
\sup_{U_\e} u_{k,\e} \leq C\ \textnormal{diam}(\Omega) \norm{F_{k,\e}}{L^n(U_\e)}.
\end{align*}
Thanks to the properties of the sequences defined, we get
\begin{align*}
\sup_{U_\e} u_{k,\e} \xrightarrow[]{k} \sup_{U_\e} u_\e
\end{align*}
and
\begin{align*}
F_{k,\e}\xrightarrow[]{k} F \spc \spc \inn L^n(U_\e)
\end{align*}
that, together with previous inequality, imply
\begin{align*}
\sup_{U_\e} u_\e \leq C\ \textnormal{diam}(\Omega) \norm{F}{L^n(U_\e)},
\end{align*}
i.e.
\begin{align*}
\sup_{U_\e} u \leq C\ \textnormal{diam}(\Omega) \norm{f}{L^n(U_\e)}+\e.
\end{align*}
Letting $\e\to 0$, thanks to the fact that $\limsup_{x\to \partial \Omega} u\leq 0$ and $U_\e \to \Omega$, we finally get
\begin{align*}
\sup_\Omega u \leq C\ \textnormal{diam}(\Omega) \norm{f}{L^n(\Omega)}.
\end{align*}
\end{proof}

\begin{remark}\label{Rmk:ABPExhaustion}
Observe that the constant $C$ in previous theorem depends on $n,\ a,\ b,\ c_0,\ C_0$ and on the family of harmonic neighbourhoods $\mathcal{W}$ that $\Omega$ intersects. In particular, by construction if $\Omega$ and $\Omega'$ are covered by the same family of harmonic neighbourhoods $\mathcal{W}$, $|\Omega|>|\Omega'|$ and $C$ and $C'$ are the constants given by Theorem \ref{Thm:ABP} on $\Omega$ and $\Omega'$ respectively, then
\begin{align*}
C>C'.
\end{align*}
As a consequence, the constant $C$ is monotone (increasing) with respect to the inclusion and so we can use the same $C=C(\Omega)$ for every subdomain $\Omega'\subseteq \Omega$.
\end{remark}

\begin{remark}
The explicit expression of the constant $C$ in \eqref{Eq:ABP1} is the following
\begin{align*}
C=\frac{t^{1/p} 2^{n/p}\left[t\left(2^{n(p+1)/p}C_{\rr^n} C_1 \right)^t +1 \right]}{\theta^{1/p}}
\end{align*}
where, denoting $r:=r_h(\overline{\Omega})$,
\begin{itemize}
\item $p=p(n,r,a,b,c_0,C_0)$ and $C_1=C_1(n,r,a,b,c_0,C_0)$ are the constants given in Theorem \ref{Thm:WeakHarnackTrudinger};
\item $C_{\rr^n}$ is the Euclidean doubling constant;
\item $\theta=1-\frac{|\Omega|}{|\widehat{\Omega}|}$;
\item $t\leq \frac{|\Omega_r| 2^{n/2}}{|\mathbb{B}_{r/8}|}$.
\end{itemize}
Observe that in the Euclidean case we have $r_h=+\infty$, implying that if $\Omega\subset \rr^n$ is a fixed bounded domain, then we can choose a radius $R=(8\ \textnormal{diam}(\Omega))$ in order to get $\Omega \subset \mathbb{B}_{R/8}$. By Remark \ref{Rmk:ABPExhaustion}, we can use the ABP constant of the domain $\mathbb{B}_{R/8}$ also for the domain $\Omega$. In particular, thanks to the Euclidean (global) doubling property, the constants $t$ and $\theta$ of the domain $B_{R/8}$ do not depend neither on $\mathbb{B}_{R/8}$ nor $\Omega$, while the constants $p$ and $C_1$ depend on $n,\ R$ (and hence on $\textnormal{diam}(\Omega)$), $b,c_0$ and $C_0$. This means that in case $M=\rr^n$ the constant in Theorem \ref{Thm:ABP} depends on the domain $\Omega$ only through its diameter. Moreover, by Remark \ref{Rmk:b=0}, this last dependence on the diameter of $\Omega$ is avoided in case $b=0$ (for instance for the Euclidean Laplacian).
\end{remark}


%

\subsection{Generalized principal eigenfunction in general bounded domains}
As already claimed, the aim of this section is to prove a maximum principle for smooth unbounded domains in a general Riemannian manifolds. While in the bounded case the validity of the maximum principle is strictly related to the positivity of the first Dirichlet eigenvalue, in unbounded domains the existence of classical principal eigenelements is not even guaranteed. In this direction, following what done by Nordman in \cite{nordmann2021maximum}, we will consider a generalization of the notion of principal eigenvalue (and related eigenfunction) in order to extend this relation to unbounded smooth domains. 

\begin{definition}
The \textnormal{generalized principal Dirichlet eigenvalue} of the operator $\operl$ acting on a (possibly nonsmooth) domain $\Omega\subset M$ is defined as
\begin{align*}
\lambda_1^{-\operl}(\Omega) := \sup\bra{\lambda \in \rr \ :\ \operl+\lambda\ admits\ a\ positive\ supersolution}
\end{align*}
where $u$ is said to be a \textnormal{supersolution} for the operator $\mathcal{L}+\lambda$ if $u\in C^{2}(\overline{\Omega})$ and it satisfies
\begin{align*}
\system{rl}{(\operl+\lambda) u \leq 0 & \inn \Omega \\ u\geq 0 & \onn \partial \Omega.}
\end{align*}
\end{definition}
Clearly, the previous definition makes sense both in bounded and unbounded domains and in the former case it coincides with the classical notion of principal eigenvalue. Moreover, if $A^{-1}\cdot B=\nabla \eta$ for a smooth function $\eta$ (for instance, if $B\equiv 0)$, then $\operl$ is symmetric on $L^2(\Omega, \textnormal{dv}_\eta)$, where $\textnormal{dv}_\eta=e^\eta \dvol$, and we have a variational characterization of $\lambda_1$ through the Rayleigh quotient
\begin{align*}
\lambda_1^{-\operl}(\Omega)= \underset{\norm{\psi}{L^2(\Omega, \textnormal{dv}_\eta)}=1}{\inf_{\psi\in H^1_0(\Omega, \textnormal{dv}_\eta)}} \left( \int_\Omega g(A\cdot \nabla \psi, \nabla\psi) \dvol_\eta-\int_\Omega c \psi^2 \dvol_\eta\right).
\end{align*}

The next step consists in proving the existence of a couple of generalized eigenelements. The first result we need is a boundary Harnack inequality, obtained adapting \cite[Theorem 1.4]{berestycki1996inequalities} to the Riemannian setting.

\begin{theorem}[Krylov-Safonov Boundary Harnack inequality]\label{Thm:KSHarnack}
Let $(M,g)$ be a complete Riemannian manifold and $\Omega\subset M$ a bounded domain with possibly nonsmooth boundary. Fix $x_0\in \Omega$ and consider $G\subset \Omega \cup \Sigma$ compact, where $\Sigma$ is a smooth open subset of $\partial \Omega$. Then, there exists a positive constant $C$, depending on $x_0,\ \Omega,\ \Sigma,\ G,\ a,\ b,\ c_0$ and $C_0$, so that for every nonnegative function $u\in W^{2,p}\loc(\Omega \cup \Sigma)$, $p>n$, satisfying
\begin{align*}
\system{ll}{\mathcal{L}u=0 & a.e.\ \inn \Omega\\ u>0 & \inn \Omega\\ u=0 & \onn \Sigma}
\end{align*}
we have
\begin{align*}
u(x)\leq C u(x_0) \spc \forall x \in G.
\end{align*}
\end{theorem}

\begin{proof}
Let $\mathcal{U}:=\bra{U_1,...,U_m}$ be a family of local charts of $M$ intersecting and covering $\partial \Omega$ and with the property that $\partial G \cap U_i$ is connected for every $i$. Fix $\e>0$ small enough so that $d^M(x_0,\partial \Omega)>2\e$,
\begin{align*}
\emptyset \neq \bra{x\in \Omega \ :\ d(x,\partial \Omega)\in (\e,2\e)}\subseteq \bigcup_{1\leq i\leq m} U_i
\end{align*}
and
\begin{align*}
\bra{x\in \Omega\ :\ d(x,\partial \Omega)>2\e} \neq \emptyset.
\end{align*}
Let $\Omega_{\e}$ a smooth subdomain of $\Omega$ satisfying
\begin{align*}
\bra{x\in \Omega\ :\ d(x,\partial \Omega)>2\e} \subseteq \Omega_{\e} \subseteq \bra{x\in \Omega\ :\ d(x,\partial \Omega)>\e}.
\end{align*}
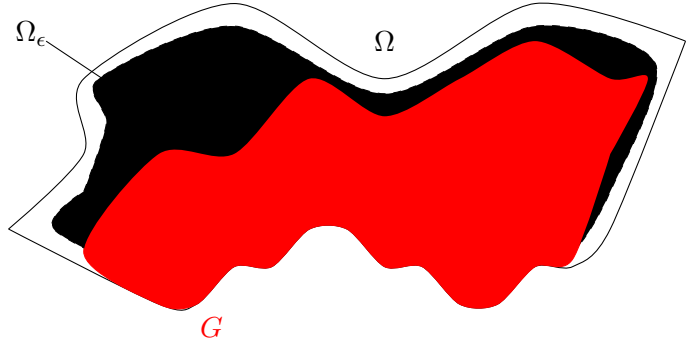
\begin{center}
\begin{tikzpicture}
\path[name path=Omega1, draw] plot [smooth] coordinates {(1,4) (2,5) (2,6) (4,7) (6,6) (8,7) (10,6.5)};
\path[name path=Omega2, draw]  plot [smooth] coordinates {(10,6.5) (9,4) (8.5,3.5) (8,3.5) (7.5, 3) (7,3) (6.5,3.5) (6,3.5) (5.5,4) (5,4) (4.5,3.5) (4,3.5) (3.5,3) (3,3) (1,4)};
\path (6,6.5) node{$\Omega$};

\path[name path=Compact, draw, red] plot [smooth] coordinates {(9,5) (8.5,3.6) (8,3.5) (7.5, 3) (7,3) (6.5,3.5) (6,3.5) (5.5,4) (5,4) (4.5,3.5) (4,3.5) (3.5,3) (3,3) (2,3.7) (3,5) (4,5) (5,6) (6,5.5) (7,6) (8,6.5) (9,6) (9.5,6) (9,5)};

\path[dashed, draw] plot [smooth] coordinates {(2,4.5) (2.3,5.4) (2.2, 6) (4,6.7) (6,5.8) (8,6.7) (9.6,6.2) (9,4.5) (8.5,3.7) (8,3.7) (7.5, 3.2) (7,3.2) (6.5,3.7) (6,3.7) (5.5,4.3) (5,4.3) (4.5,3.8) (4,3.8) (3.5,3.3) (3,3.3) (1.6,4) (2,4.5)};
\path[fill, opacity=0.1] plot [smooth] coordinates {(2,4.5) (2.3,5.4) (2.2, 6) (4,6.7) (6,5.8) (8,6.7) (9.6,6.2) (9,4.5) (8.5,3.7) (8,3.7) (7.5, 3.2) (7,3.2) (6.5,3.7) (6,3.7) (5.5,4.3) (5,4.3) (4.5,3.8) (4,3.8) (3.5,3.3) (3,3.3) (1.6,4) (2,4.5)};
\path[fill, red, opacity=0.3] plot [smooth] coordinates {(9,5) (8.5,3.6) (8,3.5) (7.5, 3) (7,3) (6.5,3.5) (6,3.5) (5.5,4) (5,4) (4.5,3.5) (4,3.5) (3.5,3) (3,3) (2,3.7) (3,5) (4,5) (5,6) (6,5.5) (7,6) (8,6.5) (9,6) (9.5,6) (9,5)};
\path[red] (3.7,2.7) node{$G$};
\path (1.3,6.6) node{$\Omega_\e$};
\path[draw] (1.5,6.5)--(3,5.5);
\path[fill] (3,5.5) circle (2pt);
\end{tikzpicture}
\end{center}

Clearly, $\partial \Omega_{\e} \subset \bigcup_{1\leq i\leq m} U_i$. Now complete $\mathcal{U}$ to a cover of $\Omega$ by coordinate neighbourhoods of $M$
\begin{align*}
\mathcal{V}=\mathcal{U}\cup \mathcal{U}'= \mathcal{U} \cup \bra{U_{m+1},...,U_h}
\end{align*}
so that
\begin{align*}
\overline{\Omega}_{\e} \subset \bigcup_{m+1\leq i\leq h} U_i \spc \spc \textnormal{and} \spc \spc \partial \Omega \cap \left(\bigcup_{m+1\leq i\leq h} U_i \right) = \emptyset.
\end{align*}
Up to considering a larger family $\mathcal{U}'$, we can suppose that for every $i=m+1,...,h$ there exists $W_i\Subset U_i$ open subset such that
\begin{align*}
\overline{\Omega}_{\e} \subset \bigcup_{m+1\leq i\leq h} W_i, \spc \spc \partial \Omega \cap \left(\bigcup_{m+1\leq i\leq h} W_i \right) = \emptyset
\end{align*}
and
\begin{align*}
W_i \cap W_j \neq \emptyset \spc \Leftrightarrow \spc U_i \cap U_j \neq \emptyset.
\end{align*}
Lastly, up to considering a larger family $\mathcal{U}$ and a smaller $\e$, we can suppose that for every $i\in \bra{1,...,m}$ there exists a compact subset $E_i\subset \left(U_i \cap \overline{\Omega} \right)$ so that
\begin{align*}
\overline{\Omega}\setminus \Omega_\e \subset \bigcup_{1\leq i\leq m} E_i
\end{align*}
and every $E_i$ intersects at least one $W_j$.

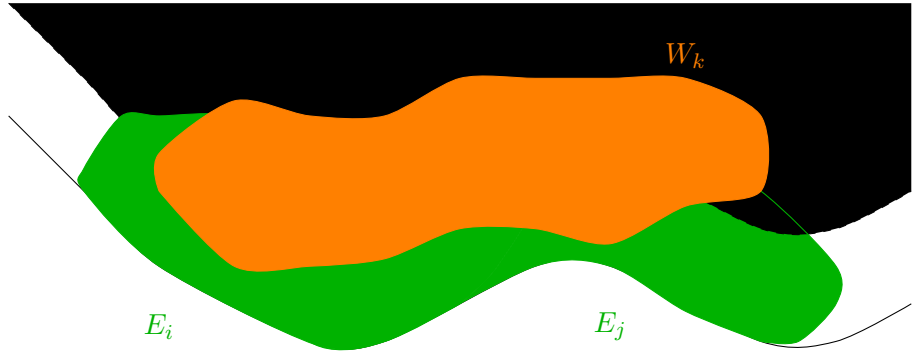
\begin{center}
\begin{tikzpicture}
\path[draw] plot [smooth] coordinates {(-2,5) (-1,4) (0,3) (2,2) (3,2) (5,3) (6,3) (8,2) (9,2) (10,2.5)};

\path[fill, white] plot [smooth] coordinates {(8,2) (8.5,2) (9,2.5) (9,3) (8,4) (7,4.5) (6,5) (5,5) (4,4.5) (3,4) (2,4) (1.6,3.5) (1.6,2.7) (2.07,1.96) (3,2) (5,3) (6,3) (7,2.4) (8,2)};
\path[fill, black!30!green, opacity=0.2] plot [smooth] coordinates {(8,2) (8.5,2) (9,2.5) (9,3) (8,4) (7,4.5) (6,5) (5,5) (4,4.5) (3,4) (2,4) (1.6,3.5) (1.6,2.7) (2.07,1.96) (3,2) (5,3) (6,3) (7,2.4) (8,2)};
\path[fill, white] plot [smooth] coordinates{(4,2.5) (4.5,3) (5,4) (4,4.5) (3,4.5) (2,4.5) (1,5) (0,5) (-0.5,5) (-1,4.3) (-1,4) (0,3) (2.07,1.96)  (3,2) (4,2.5)};
\path[draw, dashed] plot [smooth] coordinates {(-2,6.5) (-1,5.5) (0,4.5) (2,3.5) (3,3.5) (5,4.5) (6,4.5) (8,3.5) (9,3.5) (10, 4)};
\path[fill, opacity=0.1] plot [smooth] coordinates {(-2,6.5) (-1,5.5) (0,4.5) (2,3.5) (3,3.5) (5,4.5) (6,4.5) (8,3.5) (9,3.5) (10,4)}--(10,6.5)--(-2,6.5);

\path[draw, black!30!green] plot [smooth] coordinates {(4,2.5) (4.5,3) (5,4) (4,4.5) (3,4.5) (2,4.5) (1,5) (0,5) (-0.5,5) (-1,4.3) (-1,4) (0,3) (2.07,1.96)  (3,2) (4,2.5)};
\path[fill, black!30!green, opacity=0.2] plot [smooth] coordinates {(4,2.5) (4.5,3) (5,4) (4,4.5) (3,4.5) (2,4.5) (1,5) (0,5) (-0.5,5) (-1,4.3) (-1,4) (0,3) (2.07,1.96)  (3,2) (4,2.5)};
\path[black!30!green] (0,2.2) node{$E_i$};

\path[draw, black!30!green] plot [smooth] coordinates {(8,2) (8.5,2) (9,2.5) (9,3) (8,4) (7,4.5) (6,5) (5,5) (4,4.5) (3,4) (2,4) (1.6,3.5) (1.6,2.7) (2.07,1.96) (3,2) (5,3) (6,3) (7,2.4) (8,2)};
\path[black!30!green] (6,2.2) node{$E_j$};

\path[draw, orange] plot [smooth] coordinates {(0,4) (1,3) (2,3) (3,3.1) (4,3.5) (5,3.5) (6,3.3) (7,3.8) (8,4) (8,5) (7,5.5) (6,5.5) (5,5.5) (4,5.5) (3,5) (2,5) (1,5.2) (0,4.5) (0,4)};
\path[fill, orange, opacity=0.2] plot [smooth] coordinates {(0,4) (1,3) (2,3) (3,3.1) (4,3.5) (5,3.5) (6,3.3) (7,3.8) (8,4) (8,5) (7,5.5) (6,5.5) (5,5.5) (4,5.5) (3,5) (2,5) (1,5.2) (0,4.5) (0,4)};
\path[orange] (7,5.8) node{$W_k$};
\path (-0.5,6) node{$\Omega_\e$};
\end{tikzpicture}
\end{center}

For every $i=m+1,...,h$ we can apply the Euclidean version of Krylov-Safonov Harnack inequality, \cite[Corollary 8.21]{gilbarg1977elliptic}, to the couple $W_i \Subset U_i$. Let $C_i=C_i(n,U_i, b, c_0, C_0, W_i)>0$ be the corresponding constant and define
\begin{align*}
K:=\max\dbra{m+1\leq i\leq h} C_i\geq 1.
\end{align*}
If $x\in G$, we have two possible cases:
\begin{enumerate}
\item \underline{$x\in G \cap \Omega_\e$}: we can consider a sequence of distinct neighbourhoods $U_{i_1},..,U_{i_t}\in \mathcal{U}'$ so that
\begin{align*}
x\in W_{i_1}, \spc \spc x_0 \in W_{i_t} \spc \spc \textnormal{and}\\ W_{i_j} \cap W_{i_{j+1}} \neq \emptyset \ \ \forall j=1,...,t-1
\end{align*} 
and by (Euclidean) Krylov-Safonov Harnack inequality, we get
\begin{align*}
u(x) & \leq \sup\dbra{W_{i_1}} u \leq K \inf\dbra{W_{i_1}} u \leq K \inf\dbra{W_{i_1}\cap W_{i_2}} u \\
&\leq K \sup\dbra{W_{i_2}}u \leq ... \leq K^t \inf\dbra{W_{i_t}} u \leq K^t u(x_0).
\end{align*}
Since the sequence of neighbourhoods can be chosen with at most $h-m$ different elements, it follows that
\begin{align*}
u(x)\leq  \widetilde{K}\ u(x_0)
\end{align*}
where $\widetilde{K}:=K^{k-m}$ does not depend on the choice of $x\in G\cap \Omega\dbra{\e}$.

\item \underline{$x\in G\setminus \Omega\dbra{\e}$}: without loss of generality, we can suppose $x	\in U_1$. By Theorem 1.4 in \cite{berestycki1996inequalities} applied to $U_1$ and $E_1$, we get
\begin{align*}
u(x)\leq B_1\ u(z(x))
\end{align*}
where $B_1=B_1 (n,a, b, c_0, C_0,U_1,E_1)>1$ and $z(x)\in U_1\cap W_j$ for some $j\geq m+1$, up to enlarge slightly $W_j$ and $E_1$. Retracing what done in previous point, we obtain that
\begin{align*}
u(x)\leq B_1\ u(z(x)) \leq B_1 \sup\dbra{W_j} u \leq B_1\ \widetilde{K}\ u(x_0).
\end{align*}
\end{enumerate}
Choosing $B:=\max\dbra{1\leq i\leq m} B_i$ and defining $C:=B \widetilde{K}\geq \widetilde{K}$, we get
\begin{align*}
u(x)\leq C\ u(x_0)
\end{align*}
for every $x\in G$, obtaining the claim.
\end{proof}

\begin{remark}\label{Rmk:Harnack1}
Observe that $C$ actually depends only on the neighbourhoods that $G$ intersects and not really on $G$, i.e. $C$ is ``stable'' under small perturbations.
\end{remark}

Next stage consists in the construction of a function $u_0$ which vanishes at those points of $\partial \Omega$ that admit a \textit{barrier}. It will be needed to show that the generalized principal eigenfunction vanishes at smooth portions of $\partial \Omega$.

\begin{definition}
We say that $y\in \partial \Omega$ admits a \textnormal{strong barrier} if there exists $r>0$ and $h\in W^{2,n}\loc(\Omega \cap B_r(y))$ which can be extended continuously to $y$ by setting $h(y)=0$ and so that
\begin{align*}
\operm h \leq -1.
\end{align*}
\end{definition}

\begin{remark}\label{Rmk:MillerConeCondition}
As proved by Miller in \cite{miller1967barriers}, the strong barrier condition at $y\in \partial \Omega$ is implied by the exterior cone condition in any local chart, i.e. by the fact that in every local chart around $y$ there exists an exterior truncated cone $C_y$ with vertex at $y$ and lying outside $\overline{\Omega}$. In particular, on every smooth sector $\Sigma$ of $\partial \Omega$ every point $y\in \Sigma$ satisfies the (local) exterior cone condition, and thus the strong barrier condition.
\end{remark}

\begin{theorem}\label{Thm:FunctionU0}
Let $(M,g)$ be a complete Riemannian manifold. Given a (possibly nonsmooth) bounded domain $\Omega \subset M$, there exists $u_0$ positive solution to $\operm u_0=-g_0\in \mathbb{R}_{<0}$ in $\Omega$ that can be extended as a continuous function at every point $y\in \partial \Omega$ admitting a strong barrier by setting $u_0(y)=0$.
\end{theorem}

\begin{proof}
Consider $\Lambda\subset M$ a bounded, open and smooth domain containing $\overline{\Omega}$ properly and let $\mathcal{G}$ be the positive Dirichlet Green function on $\overline{\Lambda}$ associated to the differential operator $\mathcal{M}-1$. Fixed $x_0\in \Lambda\setminus \overline{\Omega}$, let $G(\cdot):=\mathcal{G}(x_0,\cdot)$ so to have
\begin{align*}
\system{ll}{\mathcal{M}G=G & \inn \Omega\\ G>0 & \inn \overline{\Omega}}
\end{align*}
and define
\begin{align*}
g_0=\min_{\overline{\Omega}} G \spc \spc \textnormal{and} \spc \spc G_0=\max_{\overline{\Omega}} G.
\end{align*}
Consider an exhaustion $\bra{H_j}_j$ of $\Omega$ by smooth nested subdomains satisfying $\overline{H}_j\subset H\dbra{j+1}$ and let $u_j$ be the solutions of
\begin{align*}
\system{ll}{\mathcal{M}u_j=-g_0 & \inn H_j\\ u_j=0 & \onn \partial H_j.}
\end{align*}
In particular, $u_j\in W\ubra{2,p}(H_j)$ for every $p>n$ and, by the standard maximum principle, $\bra{u_j}_j$ is an increasing sequence of positive functions. Moreover
\begin{align*}
\mathcal{M}(u_j+G)=-g_0+G\geq 0
\end{align*}
so, again by maximum principle, it follows that
\begin{align*}
u_j+G\leq \max_{\partial \Omega_j} G\leq G_0,
\end{align*}
i.e. $u_j\leq G_0-G\leq G_0$ for every $j$. Hence there exists a function $u_0$ so that
\begin{align*}
& u_j \rightharpoonup u_0 \spc \inn W^{2,p}(E)\\
& u_j \rightarrow u_0 \spc \inn C^1(E)
\end{align*}
for every $p>n$ and every $E\subset \Omega$ compact. Moreover, $\mathcal{M}u_0=-g_0$ and $0<u_0\leq G_0$ by construction.

The next step consists in proving that $u_0$ can be extended continuously to $0$ at every $y\in \partial \Omega$ admitting a strong barrier. Fix such a $y\in \partial \Omega$ admitting a strong barrier, i.e. so that for some $B_r(y)$ there exists in $U=B_r(y)\cap \Omega$ a positive function $h\in W\ubra{2,n}\loc(U)$ satisfying $\mathcal{M}h\leq -1$ which can be extended continuously to $y$ by imposing $h(y)=0$. Without loss of generality, we can suppose $r<\textnormal{inj}(y)$.
Let $h$ be the strong barrier associated to $y$ and choose $j$ big enough so that $V=H_j \cap B_{r/2}(y)\neq \emptyset$: choosing $\e>0$ small so that
\begin{align*}
\e \mathcal{M}\left(d(x,y)^2\right)\leq \frac{1}{2} \spc \inn U
\end{align*}
the function $\widetilde{h}=h+\e d(x,y)^2$ satisfies
\begin{align*}
\mathcal{M}\widetilde{h}\leq -\frac{1}{2} \spc \inn U.
\end{align*}
Moreover, if $d(x,y)=\frac{r}{2}$ and $x\in \overline{H}_j$, then
\begin{align*}
\widetilde{h}(x)\geq \e \frac{r^2}{4}=:\delta
\end{align*}
and, up to decrease $\e$, we can suppose $\delta\leq 1$ and that the function $w=G_0 \frac{\widetilde{h}}{\delta}-u_j$ satisfies
\begin{align*}
\system{ll}{\mathcal{M}w \leq 0 & \inn V \\ w\geq 0 & \onn \partial V.}
\end{align*}
By the Maximum Principle, it follows $w\geq 0$ in $V$, i.e.
\begin{align*}
u_j(x)\leq G_0 \frac{\widetilde{h}(x)}{\delta} \spc \inn V.
\end{align*}
Fixing $x\in H_j \cap B\dbra{r/2}(y)$ and letting $j\to +\infty$, it follows
\begin{align*}
u_0(x)\leq G_0 \frac{\widetilde{h}(x)}{\delta}.
\end{align*}
Since the previous inequality holds for every $x\in H_j\cap B\dbra{r/2}(y)$ and for every $j$ big enough, by the continuity of $\widetilde{h}$ in $y$ the claim follows. 
\end{proof}

\begin{remark}
Theorem \ref{Thm:FunctionU0} has been obtained thanks to an adaptation of the argument presented in \cite[Section 3]{MR1258192}. Unless small details, the structure of the proof remained unchanged with respect to the one by Berestycki, Nirenberg and Varadhan.
\end{remark}

Finally, we can prove the existence of a generalized principal eigenfunction in any bounded Riemannian domain

\begin{theorem}\label{Thm:BoundedPrincipalEigenfunction}
Let $(M,g)$ be a complete Riemannian manifold of dimension $\dim(M)=n$ and consider a (possibly nonsmooth) bounded domain $\Omega \subset M$. If $u_0$ is the function obtained in Theorem \ref{Thm:FunctionU0}, then
\begin{enumerate}
\item there exists a principal eigenfunction $\phi$ of $\operl$
\begin{align*}
\mathcal{L}\phi =-\lambda_1 \phi
\end{align*}
so that $\phi\in W\ubra{2,p}\loc (\Omega)$ for every $p<+\infty$;
\item normalizing $\phi$ to have $\phi(x_0)=1$ for a fixed $x_0\in \Omega$, there exists a positive constant $C$, depending only on $x_0,\ \Omega,\ a,\ b,\ c_0$ and $C_0$, so that $\phi\leq C$;
\item there exists a positive constant $E>0$ so that $\phi\leq Eu_0$.
\end{enumerate}
\end{theorem}

\begin{remark}
The proof proceeds (more or less) as in \cite[Theorem 2.1]{MR1258192}. We present it for completeness.
\end{remark}

\begin{proof}
Fix $x_0 \in \Omega$ and consider a compact subset $F\subset \Omega$ so that $x_0 \in \textnormal{int}\ F$ and $|\Omega \setminus F|=\delta$, where $\delta>0$ is a constant (small enough) to be chosen. Let $\bra{\Omega_j}_j$ be a sequence of relatively compact smooth subdomains of $\Omega$ with $F\subset \Omega_1$ and satisfying
\begin{align*}
\overline{\Omega}_i\subset \Omega_{i+1} \spc \forall i \spc \spc \textnormal{and} \spc \spc \bigcup_i \Omega_i =\Omega.
\end{align*}
By the smoothness of $\Omega_j$, for every $j$ there exists a couple of principal eigenelements $(\mu_j,\phi_j)$ for $\operl$ so that
\begin{align*}
\system{ll}{\operl \phi_j = -\mu_j \phi_j & \inn \Omega_j \\ \phi_j >0 & \inn \Omega_j\\ \phi_j=0 & \onn \partial \Omega_j}
\end{align*}
rescaled so that $\phi_j(x_0)=1$ and with  $\phi_j \in W^{1,p}(\Omega_j)$ for any $p<+\infty$. Moreover, since $\phi_k>0$ in $\overline{\Omega}_j$ for $k>j$, by the standard maximum principle it follows that $\mu_j>\mu\dbra{j+1}>\lambda_1:=\lambda_1^{-\operl}(\Omega)$ for every $j$. In particular, by monotonicity $\{\mu_j\}_j$ converges to a certain $\mu\geq \lambda_1$.

By the standard Harnack inequality applied in $\Omega_1$ it follows that there exists a positive constant $C=C(n,a,b,c_0,C_0,x_0,\Omega_1,F)$ so that
\begin{align}\label{Eq:AppBoundedPrincipalEigenfunction0}
\max_{F} \phi_j \leq C\ \phi_j (x_0)=C
\end{align}
for every $j\geq 1$. 

Now consider $U_j:=\Omega_j \setminus F$ and $v=\phi_j-C$: we have
\begin{align*}
\operm v = - c \phi_j - \mu_j \phi_j \geq -b \phi_j -\mu_j \phi_j
\end{align*}
and
\begin{align*}
\limsup_{x\to \partial U_j} v \leq 0.
\end{align*}
Now let $\Lambda$ be a smooth, bounded domain containing $\overline{\Omega}$ and let $C_{\Lambda}$ be the constant given by Theorem \ref{Thm:ABP} on $\Lambda$. Observing that $\overline{U}_j \subset \Lambda$ for every $j$, by Theorem \ref{Thm:ABP} and Remark \ref{Rmk:ABPExhaustion} it follows that
\begin{equation}\label{Eq:AppBoundedPrincipalEigenfunction}
\begin{split}
\max_{\overline{U}_j} \phi_j - C &=\max_{\overline{U}_j} v \\
& \leq C_{\Lambda}\ \textnormal{diam}(\Lambda)\ \norm{(b+\mu_j)\phi_j}{L^n(U_j)} \\
& \leq C_{\Lambda}\ \textnormal{diam}(\Lambda)\ (b+\mu_j)\ \max_{\overline{U}_j} \phi_j\ \delta^\frac{1}{n}.
\end{split}
\end{equation}
Let $B_r$ be a ball completely contained in $F$: by \cite[Lemma 6.3]{padilla1997principal} there exists a positive constant $K$, depending only on $\textnormal{dim}(M)$ and on the coefficients of $\operl$, so that
\begin{align*}
\mu_j\leq \frac{K}{r^2}.
\end{align*}
Using the previous inequality in \eqref{Eq:AppBoundedPrincipalEigenfunction}, we get
\begin{align*}
\max_{\overline{U}_j} \phi_j - C &\leq C_{\Lambda}\ \textnormal{diam}(\Lambda)\ \left(b+\frac{K}{r^2}\right)\ \max_{\overline{U}_j} \phi_j\ \delta^\frac{1}{n}
\end{align*}
and choosing  $\delta$ small enough so that
\begin{align*}
 C_{\Lambda}\ \textnormal{diam}(\Lambda)\ \left(b+\frac{K}{r^2}\right)\ \delta^\frac{1}{n}\leq \frac{1}{2}
\end{align*}
we obtain
\begin{align*}
\max_{\overline{U}_j} \phi_j \leq 2C
\end{align*}
that, together with \eqref{Eq:AppBoundedPrincipalEigenfunction0}, implies
\begin{align*}
\max_{\overline{\Omega}_j} \phi_j \leq 2C=:C.
\end{align*}
By interior $W^{2,p}$ estimates (\cite[Theorem 6.2]{gilbarg1977elliptic}), it follows that
\begin{align*}
\norm{\phi_k}{W^{2,p}(\Omega_j)}\leq C_j \spc \spc \forall k\geq j+1
\end{align*}
implying the existence of a function $\phi$, positive in $\Omega$, so that
\begin{align*}
& \phi_j \rightharpoonup \phi \spc \spc \inn W^{2,p}\loc (\Omega)\\
& \phi_j \rightarrow \phi \spc \spc \inn W^{2,\infty}\loc (\Omega).
\end{align*}
By construction, $\phi$ solves
\begin{align*}
\mathcal{L}\phi=-\mu \phi \spc \inn \Omega
\end{align*}
with $\phi(x_0)=1$ and $\phi\leq C$. Moreover, by definition of $\lambda_1$ and by the fact that $\mu\geq \lambda_1$, it follows that $\mu=\lambda_1$, obtaining the claims 1 and 2.

Lastly, observing that
\begin{align*}
\system{ll}{\mathcal{M}\phi_j=-(\mu_j+c)\phi_j\geq -(\mu_j+b)\phi_j & \inn \Omega_j\\ \phi_j=0 & \onn \partial \Omega_j}
\end{align*}
and recalling that
\begin{align*}
\system{ll}{\mathcal{M}u_0=-g & \inn \Omega \\ u_0>0 & \inn \Omega}
\end{align*}
we get
\begin{align*}
\system{ll}{\mathcal{M}\left(\phi_j- \frac{C}{g_0} (\mu_j^+ + b)u_0 \right)\geq -(\mu_j+b) C + (\mu_j^+ + b) C\geq 0 & \inn \Omega_j \\ \phi_j- \frac{C}{g_0} (\mu_j^+ + b)u_0<0 & \onn \partial \Omega_j}
\end{align*}
and, by standard maximum principle,
\begin{align*}
\phi_j \leq \frac{C}{g_0}(\mu_j^+ + b)u_0 \spc \inn \Omega_j.
\end{align*}
Letting $j\to \infty$, it follows
\begin{align*}
\phi \leq \frac{C}{g_0} (\lambda_1^+ + b)u_0=Eu_0.
\end{align*}
\end{proof}

\begin{remark}\label{Rmk:PhiVanishesBoundary}
Using remark \ref{Rmk:MillerConeCondition}, Theorem \ref{Thm:FunctionU0} and the third point of the previous theorem, we can see that the function $\phi$ vanishes on every smooth portion of $\partial \Omega$. As a consequence, if we consider a smooth domain $\Omega$ and $x_0\in \partial \Omega$, then for every $R>0$ there exists a couple of eigenelements $(\varphi^R, \lambda_1^{-\operl})$ of the following Dirichlet problem
\begin{align*}
\system{ll}{\operl\varphi^R=-\lambda_1^R \varphi^R & \inn \Omega\cap B_R(x_0)\\ \varphi^R=0 & \onn \textnormal{smooth portions of}\ \partial(\Omega \cap B_R(x_0)).}
\end{align*}
\end{remark}

\subsection{Generalized principal eigenfunction in smooth unbounded domains}
As a consequence of previous construction, we get the analogue of Theorem 1.4 in \cite{berestycki2015generalizations}. The Euclidean proof can be retraced step by step thanks to Theorem \ref{Thm:KSHarnack} and Theorem \ref{Thm:BoundedPrincipalEigenfunction}. We propose it for completeness

\begin{theorem}\label{Thm:PrincipalEigenvalueUnbounded}
Given an unbounded smooth domain $\Omega\subset M$, for any $R>0$ consider the truncated eigenvalue problem
\begin{align*}
\system{ll}{\operl\varphi^R=-\lambda_1^R \varphi^R & \inn \Omega\cap B_R\\ \varphi^R=0 & \onn \partial(\Omega \cap B_R)}.
\end{align*}
where $B_R=B_R(x_0)$ for a fixed $x_0\in \partial \Omega$.
Then:
\begin{enumerate}
\item for almost every $R>0$ there exists and is well defined the couple of eigenelemnts $(\lambda_1^R,\varphi^R)$, with $\varphi^R$ positive in $\Omega\cap B_R$;
\item $\lambda_1^R \searrow \lambda_1$ as $R\to+\infty$;
\item $\varphi^R$ converges in $C^{2,\alpha}\loc$ to some $\varphi$ principal eigenfunction of $\Omega$.
\end{enumerate}
\end{theorem}
\begin{proof}
By the smoothness of $\Omega$, for any $i\in \mathbb{N}$ there exists $r(i)\geq i$ so that $\Omega \cap B_i$ is contained in a single connected component $\Omega_i$ of $\Omega \cap B\dbra{r(i)}$. Moreover, we can suppose $\Omega_i\subset \Omega\dbra{i+1}$ for every $i$. By \cite{agmon1982positivity},
 it follows that
\begin{align*}
\lim_{i\to \infty} \lambda_1^{-\mathcal{L}}(\Omega_i)=\lambda_1^{-\mathcal{L}}(\Omega).
\end{align*}
Now fix $x_1\in \Omega_1$ and let $\varphi^i$ the generalized principal eigenfunction of $-\mathcal{L}$ in $\Omega_i$, obtained by Theorem \ref{Thm:BoundedPrincipalEigenfunction}, normalized so that $\varphi^i(x_1)=1$. Fixed $i>j\in \mathbb{N}$, since $\varphi^i \in W^{2,p}(\Omega \cap B_j)$ for every $p<+\infty$ and vanishes on $\partial \Omega \cap B_j$, by Theorem \ref{Thm:KSHarnack} with $\Omega=\Omega_{j+1}$, $\Sigma=\partial \Omega\cap B_{j+1}$ and $G=\overline{\Omega \cap B_j}$, it follows that there exists a positive constant $C_j$ so that
\begin{align*}
\sup_{\Omega\cap B_j} \varphi^i \leq C_j\ \varphi^i(x_1)=C_j \spc \spc \forall i> j.
\end{align*}
By \cite[Theorem 9.13]{gilbarg1977elliptic} it follows that $\bra{\varphi^i}_{i>j}$ are uniformly bounded in $W^{2,p}(\Omega \cap B_{j-1/2})$ for every $p<+\infty$. Thus, up to a subsequence
\begin{align*}
\varphi^i \overset{i}{\rightharpoonup} \phi_j \spc \spc \inn W^{2,p}(\Omega \cap B_{j-1/2}) \ \ \forall p<+\infty
\end{align*}
and, by \cite[Theorem 7.26]{gilbarg1977elliptic},
\begin{align*}
\varphi^i \overset{i}{\rightarrow} \phi_j \spc \spc \inn C^1(\overline{\Omega}\cap B_{j-1})
\end{align*}
to a nonnegative function $\phi_j$ that solves
\begin{align*}
\system{ll}{\mathcal{L}\phi_j=-\lambda_1^{-\mathcal{L}}(\Omega) \phi_j & \textnormal{a.e.}\ \inn \Omega \cap B_{j-1} \\ \phi_j=0 & \onn \partial \Omega \cap B_{j-1}.}
\end{align*}
By construction, $\phi_j(x_1)=1$ and so $\phi_j$ is positive in $\Omega\cap B_{j-1}$ by the strong maximum principle. Using a diagonal argument, we can extract a subsequence $\bra{\varphi^{i_k}}_{i_k}$ converging to a positive function $\varphi$ that is a solution of the above problem for all $j>1$.
\end{proof}

\subsection{Maximum principle in smooth unbounded domains}

Once that the existence of the couple of (generalized) principal eigenelements in smooth unbounded domains has been proved, we can proceed to show the validity of the maximum principle under the assumption that the generalized principal eigenvalue is positive.

In what follows we consider an operator $\operl$ of the form \eqref{Eq:Operl} and we assume that there exists a function $\eta:\Omega\to \rr$, $\eta \in C^1(\Omega) $ so that
\begin{align*}
\nabla \eta = A^{-1} \cdot B.
\end{align*}
Before proving the main result of this section, we introduce two technical lemmas

\begin{lemma}\label{Lem:LemmaNordman1}
Let $(M,g)$ be a Riemannian manifold and $\Omega\subset M$ a (possibly unbounded) smooth domain. If $v$ satisfies
\begin{align*}
\system{ll}{\mathcal{L}v\geq 0 & \inn \Omega \\ v \leq 0 & \onn \partial \Omega}
\end{align*}
and $(\lambda_1,\varphi)$ are generalized principal eigenelements of $\mathcal{L}$ on $\Omega$ with Dirichlet boundary conditions, defining $\sigma:=\frac{v}{\varphi}$ we get
\begin{align}\label{Eq:1LemmaNordman1}
\dive{\varphi^2 e^\eta A\cdot \nabla\sigma}\geq \lambda_1 e^\eta \sigma \varphi^2 \spc \spc \inn \Omega
\end{align}
and
\begin{align}\label{Eq:2LemmaNordman1}
\sigma_+ \varphi^2 g(\nu, A\cdot\nabla \sigma)=0 \spc \spc \onn \partial \Omega
\end{align}
where $\sigma_+=\max(0,\sigma)$. Since $\varphi=0$ at $\partial \Omega$, condition \eqref{Eq:2LemmaNordman1} must be understood as the limit when approaching the boundary with respect to the direction $A\cdot \nu$, where $\nu$ is the outward pointing unit vector field normal to $\partial \Omega$.
\end{lemma}

\begin{proof}
By the assumptions, it clearly follows
\begin{align*}
\dive{e^\eta A\cdot\nabla v}=e^\eta [\dive{A\cdot\nabla v}+g(B,\nabla v)]
\end{align*}
that, together with the fact that $v$ is a subsolution, implies
\begin{align*}
\dive{e^\eta A\cdot\nabla v}+e^\eta c v =e^\eta \operl v \geq 0. 
\end{align*}
Moreover, since $\varphi$ is a principal eigenfunction, we get
\begin{align*}
\dive{e^\eta A\cdot\nabla\varphi}+c\ e^\eta \varphi=-\lambda_1 e^\eta \varphi,
\end{align*}
that, using previous inequality, implies
\begin{align*}
\dive{\varphi^2 e^\eta A\cdot \nabla \sigma}&\geq \underbrace{e^\eta \left[g(\nabla \varphi, A\cdot\nabla v)-g(\nabla v, A\cdot\nabla \varphi)\right]}_{=0\ \textnormal{by the symmetry of}\ A}+v\lambda_1 e^\eta \varphi
\end{align*}
obtaining \eqref{Eq:1LemmaNordman1}.

Now let $x_0\in \partial \Omega$ and set $x_\e:=\exp_{x_0}(-\e A(x_0)\cdot \nu(x_0))$ for $\e>0$ small enough, where $\nu$ is the outward pointing unit vector field normal to $\partial \Omega$. Recalling that $v\leq 0$ at $\partial \Omega$, we have two possible cases:
\begin{enumerate}
\item \underline{$\sigma(x_\e)\leq 0$ as $\e$ becomes small}: then, $\sigma^+(x_\e)=0$ and thus \eqref{Eq:2LemmaNordman1} trivially holds in the sense of the limit for $x$ approaching the boundary of $\Omega$ along the direction $A(x_0)\cdot \nu(x_0)$. 

\item \underline{$v(x_0)=0$ and $v(x_{\e_n})>0$ for a sequence $\e_n\xrightarrow[]{n}0$}: in this case\\
\begin{center}
$g(A(x_0)\cdot \nu(x_0),\nabla v(x_0))\leq 0$
\end{center}
and, by the standard Hopf's lemma,
\begin{align*}
g(A(x_0) & \cdot\nu(x_0),\nabla\varphi(x_0))\\ &=g(A(x_0)\cdot \nu(x_0), \nu(x_0)) \ g(\nu(x_0), \nabla \varphi(x_0))>0,
\end{align*}
obtaining
\begin{align*}
\lim_{\e \to 0}\sigma(x_\e)=\frac{g(A(x_0)\cdot \nu(x_0),\nabla v(x_0))}{g(A(x_0)\cdot\nu(x_0), \nabla \varphi(x_0))}\leq 0.
\end{align*}
From the definition of $\sigma$ and the fact that $v(x_0)\leq 0$, it follows that
\begin{align*}
\varphi^2 & (x_\e) \sigma^+(x_\e) g\left(\nu(x_0),A(x_0)\cdot\nabla\sigma(x_\e)\right)\\
&=[g\left(A(x_0)\cdot\nu(x_0),\nabla v(x_\e)\right) \\ &\quad \quad-\sigma(x_\e) g\left(A(x_0)\cdot\nu(x_0),\nabla\varphi(x_\e)\right)]\underbrace{v^+(x_\e)}_{\xrightarrow[]{\e\to 0}0} \\
&\xrightarrow[]{\e\to 0}0
\end{align*}
implying the claim.
\end{enumerate}
\end{proof}

Now consider the sequence of cut-off functions $\bra{\rho_k}_k\subset C^\infty_c(M)$ satisfying
\begin{align}\label{Eq:CutOff0}
\system{l}{0\leq \rho_k \leq 1\\ \norm{\nabla \rho_k}{L^\infty(M)}\xrightarrow[]{k}0\\ \rho_k \nearrow 1.}
\end{align}
For a reference, see \cite{pigola2016smooth}. Without loss of generality we can suppose $\bra{\rho_k \neq 0}\cap \partial \Omega \neq \emptyset$ for every $k$.

\begin{lemma}\label{Lem:LemmaNordman2}
Let $(M,g)$ be a Riemannian manifold and $\Omega\subset M$ a (possibly unbounded) smooth domain. Supposing $\lambda_1:=\lambda_1^{-\operl}(\Omega)\geq 0$, we have
\begin{align*}
\lambda_1 \int_\Omega \rho_k^2 e^\eta (v^+)^2 \dvol\leq \int_\Omega g(\nabla\rho_k, A\cdot \nabla \rho_k) e^\eta (v^+)^2\dvol
\end{align*}
for every $k$, where $\{\rho_k\}_k\subset C^\infty_c(M)$ is a sequence of cut-off functions satisfying \eqref{Eq:CutOff0} and so that $\bra{\rho_k \neq 0}\cap \partial \Omega \neq \emptyset$.
\end{lemma}

\begin{proof}Fix $k\in \mathbb{N}$ and let $U_k\subset\subset M$ be an open domain so that
\begin{itemize}
    \item $\textnormal{supp}(\rho_k)\subset U_k$;
    \item $\Sigma_k:=U_k \cap \partial \Omega$ is smooth (possibly not connected).
\end{itemize}
Let $\nu$ be the outward pointing unit vector field normal to $\partial \Omega$ and, for $\epsilon>0$ small enough, define
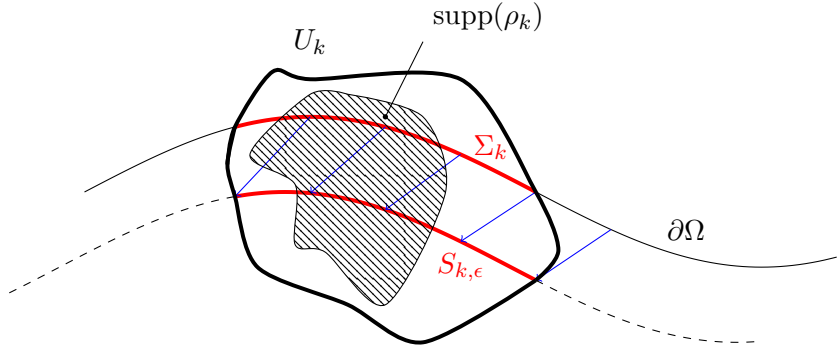
\begin{align*}
   S_{k,\epsilon}:=\left\{y\in U_k \cap \Omega \ :\ y=\textnormal{exp}_x\left(-\epsilon A(x)\cdot \nu(x) \right)\ \textnormal{for}\ x\in \partial \Omega \right\}.
\end{align*}

\begin{center}
    \begin{tikzpicture}
    
    \path[draw, smooth, domain=0:2] plot({\x},{sin(30*\x)});
    \path[draw, smooth, domain=6:10] plot({\x},{sin(30*\x)});
    \path[draw, red, smooth, line width=1.5 pt, domain=2:6] plot({\x},{sin(30*\x)});
    \path[draw, dashed, smooth, domain=0:3] plot({\x-1},{sin(30*\x-20)-1});
    \path[draw, dashed, smooth, domain=7:10] plot({\x-1},{sin(30*\x-20)-1});
    \path[draw, red, line width=1.5 pt, smooth, domain=3:7] plot({\x-1},{sin(30*\x-20)-1});
    \foreach \x in {3,4,5,6,7}{\draw[->, blue] ({\x},{sin(30*\x)})--({\x-1},{sin(30*\x-20)-1});}
    \path[draw, line width=1.5 pt] plot [smooth] coordinates{({7-1},{sin(30*7-20)-1}) (6.3,{sin(30*7-20)-0.6}) ({6},{sin(30*6)}) (5,1.5) (3,1.5) (2.5,1.6) ({2},{sin(30*2)}) (1.9,{sin(30*2)-0.5}) };
    \path[draw, line width=1.5 pt] plot [smooth] coordinates{(2,{sin(30*2)}) (1.9,{sin(30*2)-0.5}) (2,{sin(30*3-20)-1}) (2.3, -1) (3.5,-1.5) (4.5, -2) ({7-1},{sin(30*7-20)-1})};
    \path[draw, red] (5.4,0.6) node{$\Sigma_k$};
    \path[draw, red] (5,-1) node{$S_{k,\e}$};
    \path[draw] (3,2) node{$U_k$} (8,-0.5) node{$\partial \Omega$};
    \path[draw] plot [smooth] coordinates{(3.5,1.3) (3,1.3) (2.2,0.5) (2.8,0.1) (2.8,-0.7) (3.2,-0.8) (4, -1.5) (4.8,0) (4.5,1) (3.5,1.3)};
    \path[fill,pattern=north west lines, opacity=0.7] plot [smooth] coordinates{(3.5,1.3) (3,1.3) (2.2,0.5) (2.8,0.1) (2.8,-0.7) (3.2,-0.8) (4, -1.5) (4.8,0) (4.5,1) (3.5,1.3)};
    \path[draw] (4,1)--(4.5,2) (4.5,2) node[anchor=south west] {$\textnormal{supp}(\rho_k)$};
    \path[fill] (4,1) circle(1pt);
    \end{tikzpicture}
\end{center}

Next step consists in proving that there exists $\epsilon_k>0$ so that $S_{k,\epsilon}$ is a (possibly not connected) smooth hypersurface of $\Omega$ for every $0\leq \epsilon \leq \epsilon_k$. To this aim, let $p\in M$ and define $O_p \subset T_p M$ as the set of vectors $X_p$ such that the length $l_{X_p}$ of the geodesic whose initial data is $(p,X_p)$ is greater than 1 . Observe that if $\alpha \in \mathbb{R}_{>0}$, then $l_{\alpha X_p}=\alpha^{-1} l_{X_p}$ and hence
\begin{align*}
    X_p\in O_p \quad \Rightarrow \quad tX_p\in O_p\ \forall t \in (0,1].
\end{align*}
Set $O:=\cup_{p\in M} O_p$ and observe that the exponential map is smooth on $O$ (\cite[Lemma 5.2.3]{MR3469435}). 

Now fix $p\in \partial \Omega$. Since $A(p)$ is nonsingular and linear, the differential of the map $\exp_p \circ A(p):O_p\cap N_p \partial \Omega \to M$ evaluated in $0_p\in O_p$ is nonsingular and it is given by
\begin{align*}
    d_{0_p}(\exp_p \circ A(p))=\underbrace{d_{0_p} \exp_p}_{=Id} \circ\ d_{0_p} A(p)=A(p).
\end{align*}
Retracing the proofs Proposition 5.5.1 and Corollary 5.5.3 in \cite{MR3469435}, we obtain that there exists an open neighbourhood $W$ of the zero section in $N \partial \Omega$ (the normal bundle of $\partial \Omega$) on which $F:=\exp\circ A$ is a diffeomorphism onto its image. In particular, there exists a continuous function $\epsilon:\partial \Omega \to \mathbb{R}_{>0}$ so that
\begin{align*}
    (p,-t\nu(p))\in W \quad \forall t \in [0,\epsilon(p)]
\end{align*}
(see the proof of \cite[Corollary 5.5.2]{MR3469435}). Now consider a neighbourhood $V_k\subset \subset M$ of $U_k$ that intersects $\partial \Omega$ smoothly and so that for
\begin{align*}
    \epsilon_k:= \min_{p\in \overline{V_k}} \epsilon(p)
\end{align*}
we have
\begin{align*}
    Z_{k,\epsilon}:=\left\{(p,-\epsilon \nu(p))\ :\ p\in V_k \cap \partial \Omega\right\}\subset W \quad \forall \epsilon \in [0,\epsilon_k].
\end{align*}
Moreover, up to enlarge $V_k$, we have
\begin{align*}
    S_{k,\epsilon}=\left(\exp\circ A\right)(Z_{k,\epsilon})\cap U_k.
\end{align*}
Since $V_k \cap \partial \Omega$ (and hence $Z_{k,\epsilon}$) is smooth and $\left(exp\circ A\right)\Big|_{Z_{k,\e}}$ is a diffeomorphism onto its image, it follows that $S_{k,\epsilon}=\left(\exp\circ A\right)(Z_{k,\epsilon})\cap U_k$ is a smooth (possibly not connected) hypersurface for every $\epsilon \in [0,\epsilon_k]$.

Now define
\begin{align*}
    \Omega_{k,\e}:=[\Omega \cap U_k] \setminus \bigcup_{0<t<\e} S_{\e,k}
\end{align*}
and, up to decrease $\e_k$, suppose
\begin{align*}
    \Omega_{k,\e}\neq \emptyset \quad \forall \e \in [0,\e_k].
\end{align*}
By construction
\begin{align*}
    \bigcup_{0<\e<\e_k} \Omega_{\e,k}=\Omega \cap U_k.
\end{align*}

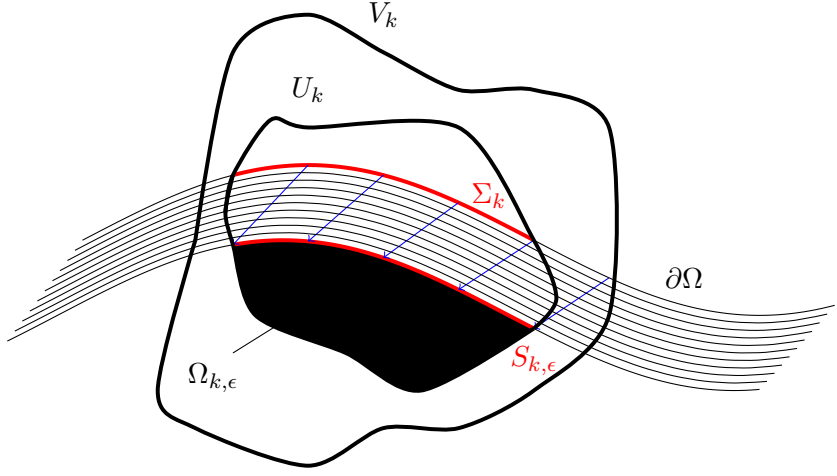
\begin{center}
    \begin{tikzpicture}
    \path[fill, opacity=0.3, domain=3:7, smooth, variable=\x] plot({\x-1},{sin(30*\x-20)-1}) --  plot [smooth] coordinates{ ({7-1},{sin(30*7-20)-1}) (4.5, -2) (3.5,-1.5)  (2.3, -1) (2,{sin(30*3-20)-1})};
    \path[draw, smooth, domain=0:2] plot({\x},{sin(30*\x)});
    \path[draw, smooth, domain=6:10] plot({\x},{sin(30*\x)});
    \path[draw, red, smooth, line width=1.5 pt, domain=2:6] plot({\x},{sin(30*\x)});
    \path[draw, smooth, opacity=0.5, domain=0:3] plot({\x-1},{sin(30*\x-20)-1});
    \path[draw, smooth, opacity=0.5, domain=7:10] plot({\x-1},{sin(30*\x-20)-1});
    
    \path[draw, red, line width=1.5 pt, smooth, domain=3:7] plot({\x-1},{sin(30*\x-20)-1});
    \foreach \x in {3,4,5,6,7}{\draw[->, blue] ({\x},{sin(30*\x)})--({\x-1},{sin(30*\x-20)-1});}
    \foreach \y in {0.1,0.2,0.3,0.4,0.5,0.6,0.7,0.8,0.9}{
        \path[draw, smooth, domain=0:10, opacity=0.5] plot({\x-1*\y},{sin(30*\x-20*\y)-1*\y});
    }
    \path[draw, line width=1.5 pt] plot [smooth] coordinates{({7-1},{sin(30*7-20)-1}) (6.3,{sin(30*7-20)-0.6}) ({6},{sin(30*6)}) (5,1.5) (3,1.5) (2.5,1.6) ({2},{sin(30*2)}) (1.9,{sin(30*2)-0.5}) };
    \path[draw, line width=1.5 pt] plot [smooth] coordinates{(2,{sin(30*2)}) (1.9,{sin(30*2)-0.5}) (2,{sin(30*3-20)-1}) (2.3, -1) (3.5,-1.5) (4.5, -2) ({7-1},{sin(30*7-20)-1})};
    \path[draw, red] (5.4,0.6) node{$\Sigma_k$};
    \path[draw, red] (6,-1.6) node{$S_{k,\e}$};
    \path[draw] (3,2) node{$U_k$} (8,-0.5) node{$\partial \Omega$};
    \path[draw] (2,-1.5)--(2.8,-1);
    \path[fill] (2.8,-1) circle(1pt);
    \path[draw] (2.2,-1.5) node[anchor=north east]{$\Omega_{k,\e}$};
    \path[draw, line width=1.5 pt] plot [smooth] coordinates{(1.5,0) (1,-2) (1.5,-2.5) (3,-3) (4,-2.5) (5,-2.5) (6,-2) (7,-1) (7,1.5) (6,2) (5,2) (4,2.5) (3,3) (2,2.5) (1.5,0)};
    \path[draw] (4,3) node{$V_k$};
    \end{tikzpicture}
\end{center}
Multiplying (3.6) by $\sigma^+ \rho_k^2$ and integrating over $\Omega_{\e,k}$, by the divergence theorem we get
\begin{align*}
\int_{\partial \Omega_{\e,k}} \sigma^+ & \rho_k^2 e^\eta \varphi^2 g(\nu,A\cdot\nabla \sigma) - \int_{\Omega_{\e,k}} g\left(\nabla\left( \sigma^+\rho_k^2\right),A\cdot\nabla \sigma\right) e^\eta \varphi^2 \\
&\geq \lambda_1 \int_{\Omega_{\e,k}} e^\eta \varphi^2 (\sigma^+)^2 \rho_k^2.
\end{align*}
Observe that
\begin{align*}
\int_{\partial \Omega_{\e,k}} \sigma^+ \rho_k^2 e^\eta \varphi^2 g(\nu,A\cdot\nabla \sigma)=\int_{S_{\e,k}\cap \textnormal{supp}(\rho_k)} \sigma^+ \rho_k^2 e^\eta \varphi^2 g(\nu,A\cdot\nabla \sigma)
\end{align*}
since $\rho_k\equiv 0$ on $\partial \Omega_{\e,k}\setminus \left(S_{\e,k}\cap \textnormal{supp}(\rho_k)\right)$. Moreover,
\begin{align*}
g\left(\nabla\left(\rho_k^2 \sigma^+\right),A\cdot\nabla \sigma\right) \geq - g\left(\nabla\rho_k,A\cdot \nabla\rho_k \right) (\sigma^+)^2,
\end{align*}
obtaining
\begin{align}\label{Eq:Inequality0}
\int_{\partial \Omega_{\e,k}} \sigma^+ & \rho_k^2 e^\eta \varphi^2 g(\nu,A\cdot\nabla \sigma)+ \int_{\Omega_{\e,k}} g\left(\nabla\rho_k,A\cdot \nabla\rho_k\right) (\sigma^+)^2 e^\eta \varphi^2 \\
& \geq \lambda_1 \int_{\Omega_{\e,k}} e^\eta \varphi^2 (\sigma^+)^2 \rho_k^2.
\end{align}
The next step is to study the behaviour of previous integrals as $\e\to 0$. Since
\begin{align*}
0\leq \lambda_1 e^\eta \varphi^2 (\sigma^+)^2 \rho_k^2 \chi_{\Omega_{\e,k}}\leq \lambda_1 e^\eta \varphi^2 (\sigma^+)^2 \rho_k^2
\end{align*}
and
\begin{align*}
\lambda_1 e^\eta \varphi^2 (\sigma^+)^2 \rho_k^2 \chi_{\Omega_{\e,k}}\to \lambda_1 e^\eta \varphi^2 (\sigma^+)^2 \rho_k^2 \quad \textnormal{a.e. in $\Omega$ as} \ \e \to 0,
\end{align*}
by dominated convergence theorem we get
\begin{align}\label{Eq:Limit1}
\lambda_1 \int_{\Omega_{\e,k}} e^\eta \varphi^2 (\sigma^+)^2 \rho_k^2=\lambda_1 \int_{\Omega} e^\eta \varphi^2 (\sigma^+)^2 \rho_k^2 \chi_{\Omega_{\e,k}} \xrightarrow[]{\e \to 0}\lambda_1 \int_{\Omega} e^\eta \varphi^2 (\sigma^+)^2 \rho_k^2.
\end{align}
Similarly, using the fact that $A$ is positive definite, we obtain
\begin{align}\label{Eq:Limit2}
 \int_{\Omega_{\e,k}} g\left(\nabla\rho_k,A\cdot \nabla\rho_k\right) (\sigma^+)^2 e^\eta \varphi^2 \xrightarrow[]{\e \to 0}  \int_{\Omega} g\left(\nabla\rho_k,A\cdot \nabla\rho_k\right) (\sigma^+)^2 e^\eta \varphi^2.
\end{align}
Lastly, for $F:=\sigma^+ \rho_k^2 e^\eta \varphi^2 g(\nu,A\cdot\nabla \sigma)$ we have
\begin{align*}
\int_{\partial \Omega_{\e,k}} F(y)=\int_{S_{k,\e}} F(y)=\int_{\partial \Omega} F\left(\textnormal{exp}_x(-\e A(x)\cdot \nu(x))\right)
\end{align*}
and for every $x\in \partial\Omega$
\begin{align*}
F\left(\textnormal{exp}_x(-\e A(x)\cdot \nu(x))\right)\xrightarrow[]{\e \to 0} 0
\end{align*}
by \eqref{Eq:2LemmaNordman1}.
Using the dominated convergence theorem, we get
\begin{align}\label{Eq:Limit3}
\int_{\partial \Omega_{\e,k}} \sigma^+ \rho_k^2 e^\eta \varphi^2 g(\nu,A\cdot\nabla \sigma)=\int_{\partial \Omega_{\e,k}} F(y)\xrightarrow[]{\e \to 0} 0.
\end{align}
Letting $\e\to 0$ in \eqref{Eq:Inequality0} and using \eqref{Eq:Limit1}, \eqref{Eq:Limit2} and \eqref{Eq:Limit3}, it follows that
\begin{align*}
\int_{\Omega} g\left(\nabla\rho_k,A\cdot \nabla\rho_k\right) (\sigma^+)^2 e^\eta \varphi^2 \geq \lambda_1 \int_{\Omega} e^\eta \varphi^2 (\sigma^+)^2 \rho_k^2,
\end{align*}
obtaining the claim, since $\sigma^+\varphi=v^+$.
\end{proof}

We are finally ready to prove the main theorem of this section.

\begin{theorem}[Unbounded Maximum Principle]\label{Thm:UnboundedMaximumPrinciple}
Let $(M,g)$ be a complete Riemannian manifold and $\Omega\subset M$ a (possibly unbounded) smooth domain. If $\lambda_1^{-\operl}(\Omega)>0$, then every function $u\in C^{2}(\overline{\Omega})$ that satisfies
\begin{align*}
\system{ll}{\operl u \geq 0 & \inn \Omega \\ u\leq 0 & \onn \partial \Omega \\ \sup_{\Omega} u <+\infty}
\end{align*}
is nonpositive.
\end{theorem}
\begin{proof}
Let $u$ be a $\operl$-subsolution with $u\leq 0$ at $\partial \Omega$ and suppose by contradiction that $u^+\not\equiv 0$. By Lemma \ref{Lem:LemmaNordman2}
\begin{align*}
\lambda_1\leq \frac{\int_\Omega g\left(\nabla \rho_k,A\cdot \nabla \rho_k\right) e^\eta (u^+)^2}{\int_\Omega \rho_k^2 e^\eta (u^+)^2}.
\end{align*}
Now consider the bounded function $w=e^{\eta/2}u^+$. We get
\begin{align*}
\frac{g(\nabla \rho_k,A\cdot \nabla \rho_k) w^2}{\int_{\Omega} \rho_k^2 w^2} & \leq C_0 \frac{g(\nabla \rho_k,\nabla \rho_k) w^2}{\int_\Omega \rho_k^2 w^2} \leq C_0 \frac{\norm{\nabla \rho_k}{L^\infty(M)}^2 w^2}{\int_{\Omega} \rho_k^2 w^2}
\end{align*}
Since $\norm{\nabla \rho_k}{L^\infty(M)}\xrightarrow{k}0$, up to extract a subsequence we can suppose $\norm{\nabla \rho_k}{L^\infty(M)}\searrow 0$, obtaining that the sequence
\begin{align*}
\left\{\frac{\norm{\nabla \rho_k}{L^\infty(M)}^2 w^2}{\int_{\Omega} \rho_k^2 w^2}\right\}_k
\end{align*}
is nonincreasing and converges to 0 almost everywhere. By the monotone convergence theorem, we get
\begin{align*}
\lambda_1 \leq \int_\Omega \frac{ g(\nabla \rho_k,A\cdot \nabla \rho_k) e^\eta (v^+)^2}{\int_\Omega \rho_k^2 e^\eta (v^+)^2} \leq C_0 \int_{\Omega} \frac{\norm{\nabla \rho_k}{L^\infty(M)}^2 w^2}{\int_{\Omega} \rho_k^2 w^2}\xrightarrow{k\to +\infty}0
\end{align*}
obtaining a contradiction.
\end{proof}

\section{Some applications of the maximum principle in unbounded domains}
Now we are going to apply Theorem \ref{Thm:UnboundedMaximumPrinciple} to generalize the symmetry results contained in \cite{bisterzo2022symmetry}.

\subsection{Strongly stable solutions in homogeneous domains} 
To start, consider a complete Riemannian manifold $(M,g)$. We recall that an \textit{isoparametric domain} $\Omega\subseteq M$ is a domain endowed by a singular Riemannian foliation $\overline{\Omega}=\bigcup_t \Sigma_t$ whose regular leaves are connected parallel hypersurfaces with constant mean curvature $H^{\Sigma_t}$. Now let $\Psi:M\to \rr$ be a smooth function and consider the weighted Riemannian manifold $M_\Psi:=(M,g,\textnormal{dv}_\Psi)=(M,g,e^\Psi \textnormal{dv})$. We say that $\Omega\subseteq M_\Psi$ is a $\Psi$-\textit{isoparametric domain} if $\overline{\Omega}$ is foliated by parallel hypersurfaces $\Sigma_t$ of constant weighted mean curvature, i.e. so that
\begin{align*}
H_\Psi^{\Sigma_t}=H^{\Sigma_t}-g(\nabla \Psi, \vec{\nu})\equiv const.
\end{align*}
where $\vec{\nu}$ is the unit vector field normal to $\Sigma_t$. Lastly, we say that $\Omega\subseteq M$ is an \textit{homogeneous domain} if $\Omega$ is an isoparametric domain whose regular leaves are orbits of the action of a closed subgroup of $\textnormal{Iso}_0(M)$, the identity component of $\textnormal{Iso}(M)$.

\begin{definition}[$\Psi$-homogeneous domain]
Given a weighted Riemannian manifold $M_\Psi$, we say that $\Omega\subseteq M_\Psi$ is a $\Psi$-\textnormal{homogeneous domain} if it is a $\Psi$-isoparametric domain and a homogeneous domain simultaneously. 
\end{definition}

For further details about isoparametric and homogeneous domains, see \cite{bisterzo2022symmetry}. We only recall that
\begin{itemize}
\item given an homogeneous domain $\overline{\Omega}$ of a complete Riemannian manifold $M$, there always exists a (finitely generated) integral distribution $\bra{X_1,...,X_k}$ of Killing vector fields of $M$ spanning pointwise every tangent space to al leaves $\Sigma_t$ of the foliation of $\overline{\Omega}$;

\item if $\overline{\Omega}$ is homogeneous and $\Psi:M\to \rr_{>0}$ is a symmetric (at least on $\overline{\Omega}$) smooth weight, then the symmetry of $\Psi$ turns $\overline{\Omega}$ into e $\Psi$-homogeneous domain.
\end{itemize}

Before proceeding with the first symmetry result of this section, we recall that on the weighted manifold $M_{\Psi}$ we have a natural counterpart to the standard Laplacian. It is the {\it weighted Laplacian}, also called {\it $\Psi$-Laplacian}, which is defined by the formula
\[
\Delta_{\Psi} u = e^{\Psi} \textnormal{div}(e^{-\Psi} \nabla u) = \Delta u - g(\nabla \Psi,\nabla u).
\]

\noindent We also recall that

\begin{definition}
The function $u\in C^3(\Omega)\cap C^1(\overline{\Omega})$ is said to be a \textnormal{stable} (respectively \textnormal{strongly stable}) \textnormal{solution} to \eqref{Eq:SystemSymmetry} if
\begin{align*}
\lambda_1^{-\Delta_\Psi + f'(u)}(\Omega):= \underset{\varphi \not \equiv 0}{\inf_{\varphi \in C^\infty_c(\Omega),}} \frac{\int_\Omega \left(|\nabla \varphi|^2+f'(u) \varphi^2 \right) \textnormal{dv}_\Psi}{\int_\Omega \varphi^2 \textnormal{dv}_\Psi}\geq 0 \spc (\textnormal{resp.}\ >0).
\end{align*}
\end{definition}

\begin{definition}
If $\overline{\Omega}$ is a weighted $\Psi$-homogeneous domain with soul $P$ inside the weighted manifold $M_\Psi$, let $d:M\to \rr_{\geq 0}$ as $x\mapsto \textnormal{dist}(x,P)$. A function $u$ on $\overline{\Omega}$ is said to be
\begin{itemize}
\item \textnormal{symmetric} if there exists a function $\widehat{u}:\rr \to \rr$ so that
\begin{align*}
u(x)=\widehat{u}(d(x));
\end{align*}
\item \textnormal{locally symmetric} if $u\in C^1(\overline{\Omega})$ and $X(u)\equiv 0$ for any smooth vector field $X\in \mathcal{D}$.
\end{itemize}
\end{definition}
\noindent By \cite[Lemma 3.7]{bisterzo2022symmetry}, these notions of symmetry coincide in our $\Psi$-homogeneous setting.
\bigskip

The first theorem, stated below, provides an adaptation of the symmetry result \cite[Theorem 5.1]{bisterzo2022symmetry} to (possibly) noncompact $\Psi$-homogeneous domains. To achieve this goal we make use of Theorem \ref{Thm:UnboundedMaximumPrinciple} in order to replace the nodal domain theorem used by the author and S. Pigola in \cite{bisterzo2022symmetry}. However, this leads to more restrictive assumptions on the solution, namely that it has to be \textit{strongly stable}.

\begin{theorem}\label{Thm:UnboundedSymmetryResult}
Let $\overline{\Omega}$ be a (possibly noncompact) $\Psi$-homogeneous domain with soul $P$ inside the weighted manifold $M_{\Psi}$. Moreover, assume that $\Psi$ is symmetric (at least on $\overline{\Omega}$) and denote with $\mathcal{D}=\{X_1,...,X_k\}$ the integrable distribution of Killing vector fields associated to the foliation of $\overline{\Omega}$.

Then, every strongly stable solution $u\in C^3(\Omega)\cap C^1(\overline{\Omega})\cap W^{1,1}(\Omega)$ to
\begin{align}\label{Eq:SystemSymmetry}
\system{ll}{\Delta_\Psi u=f(u) & in\ \Omega \\ u=c_j & on\ (\partial \Omega)_j}
\end{align}
 so that
 \begin{align*}
 \sup_\Omega X_\alpha (u)<+\infty\quad and\quad u|X_\alpha|\in L^1(\Omega,\textnormal{dv}_\Psi)\quad for\ every\ \alpha \in A 
 \end{align*}
is symmetric.
\end{theorem}


\begin{remark}
In \cite[Theorem 5.1]{bisterzo2022symmetry} the authors proved a symmetry result for (regular enough) stable solutions to
\begin{align*}
\system{ll}{\Delta_\Psi u=f(u) & \inn \Omega \\ u=c_j & \onn (\partial \Omega)_j}
\end{align*}
in case $\overline{\Omega}$ is a compact $\Psi$-homogeneous domain with associated Killing distribution $\{X_\alpha\}_{\alpha \in A}$ and the weight $\Psi$ satisfies the compatibility condition
\begin{align*}
g(X_\alpha,\Psi)\equiv const. \quad \onn \Omega \quad \forall \alpha \in A.
\end{align*}
In fact, the preceding compatibility condition implies that the weight has to be symmetric (at least on $\Omega$). Indeed, if $X_\alpha \in \mathcal{D}$, denoting $C_\alpha:=g(X_\alpha,\nabla \Psi)$ we have
\begin{align*}
\int_\Omega \textnormal{div}(\Psi X_\alpha) \dvol &= \int_\Omega \underbrace{g(X_\alpha,\nabla \Psi)}_{C_\alpha} \dvol+ \int_\Omega \Psi \underbrace{\dive{X_\alpha}}_{=0} \dvol\\
&= C_\alpha |\Omega|,
\end{align*}
while, by the divergence theorem,
\begin{align*}
\int_\Omega \textnormal{div}(\Psi X_\alpha) \dvol = \int_{\partial \Omega} \Psi\ \underbrace{g(X_\alpha, \nu)}_{=0}\dvol=0,
\end{align*}
where $\nu$ denotes the unit vector field normal to $\partial \Omega$. Putting together previous equality, we obtain $C_\alpha=0$ for every $\alpha\in A$. By previous remark, this exactly means that $\Psi$ is symmetric.
\end{remark}

\begin{proof}[Proof of Theorem \ref{Thm:UnboundedSymmetryResult}]
Let $X=X_j\in \mathcal{D}$ and define
\begin{align*}
v:=X(u).
\end{align*}
Since $u$ is locally constant on $\partial \Omega$ and $X|_{\partial\Omega}$ is tangential to $\partial \Omega$, we have
\begin{align*}
v=0 \spc \onn \partial \Omega.
\end{align*}
By \cite[Lemma 5.4]{bisterzo2022symmetry}
\begin{align*}
\Delta_\Psi v=f'(u) v
\end{align*}
implying that $v\in C^2(\Omega)$ is a solution of
\begin{align*}
\system{ll}{\left(\Delta_\Psi-f'(u)\right) v=0 & \inn \Omega\\ v=0 & \onn \partial \Omega\\ \sup_\Omega v<+\infty.}
\end{align*}
and, since $\lambda_1^{-\Delta_\Psi+f'(u)}(\Omega)>0$, by Theorem \ref{Thm:UnboundedMaximumPrinciple} 
\begin{align}\label{Eq:VLeq0}
v\leq 0 \spc \inn \Omega.
\end{align}
Let $Z:=uX$: since $X$ is Killing, it follows that $\textnormal{div} X=0$ implying
\begin{align*}
	\textnormal{div}_\Psi Z =e^\Psi \textnormal{div}(e^{-\Psi}Z)=v-\underbrace{g(\nabla \Psi,X)}_{=0}u=v \spc \in L^1(M,\textnormal{dv}_\Psi)
\end{align*}
and, by the fact that $X_x$ is tangential to $\Sigma_{d(x)}$,
\begin{align*}
	g(Z,\nu)=0
\end{align*}
for $\nu$ unit vector field normal to $\partial \Omega$. Applying the Stokes theorem by Gaffney (\cite[Theorem 4.8]{bisterzo2022symmetry}), we get
\begin{align*}
	\int_\Omega v\ \textnormal{dv}_\Psi & =\int_\Omega \textnormal{div}_\Psi Z\ \textnormal{dv}_\Psi\\
	&= \int_{\partial \Omega} g(Z,\nu)\ \textnormal{da}_\Psi=0
\end{align*}
that, together with \eqref{Eq:VLeq0}, implies $v=0$ in $\overline{\Omega}$.

We have thus proved that $X_\alpha(u)\equiv 0$ in $\overline{\Omega}$ for every $\alpha\in A$. Thanks to the fact that $\mathcal{D}$ generates every tangent space to all leaves, it follows that $u$ is locally symmetric, and hence symmetric, on $\overline{\Omega}$.
\end{proof}

\subsection{Strongly stable solutions in non-homogeneous domains in warped product manifolds}
Now consider a weighted warped product manifold
\begin{align*}
M_\Psi=(I\times_\sigma N)_\Psi
\end{align*}
where $I\subseteq \rr$ is an interval, $(N,g^N)$ is a (possibly noncompact) Riemannian manifold without boundary and $\Psi$ is a smooth weight function of the form
\begin{align*}
\Psi(r,\xi)=	\Phi(r)+\Gamma(\xi)
\end{align*}
for $(r,\xi)\in I\times N$. The second result we want to deal with concerns the case when the domain is an annulus in $\overline{A}(r_1,r_2)\subseteq M$ and there are not enough Killing vector fields tangential to $N$ (and thus there are not enough local isometries acting on the leaves of the annulus).

Despite this lack of symmetries on the domain, in \cite[Theorem 6.5]{bisterzo2022symmetry} the authors showed that, requiring the finiteness of $\textnormal{vol}_\Gamma (N)$, some potential theoretic tools can be used to recover a symmetry result under a stability-like assumption on the solution. More in details, they showed that if $f'(t)\leq 0$ and $u$ is a solution to
\begin{align*}
\system{ll}{\Delta_\Psi u=f(u) & \inn A(r_1,r_2) \\ u\equiv c_1 & \onn \{r_1\}\times N \\  u\equiv c_2 & \onn \{r_2\}\times N }
\end{align*}
so that $\norm{u}{C^2_{rad}}<+\infty$ and $f'(u)\geq -B$ for some nonnegative constant $B$ satisfying
\begin{align}\label{Eq:ConditionB}
0\leq B < \left(\int_{r_1}^{r_2} \frac{\int_{r_1}^s e^{-\Phi(z) \sigma^{m-1}(z) \dint{z}}}{e^{-\Phi(s)} \sigma^{m-1}(s)} \dint{s} \right)^{-1},
\end{align}
then $u(r,\xi)=\widehat{u}(r)$ is symmetric.

\begin{remark}
As already observed by the authors, as a consequence of condition \eqref{Eq:ConditionB} we get the existence of a positive smooth supersolution of the stability operator $-\Delta_\Psi+f'(u)$ in $\textnormal{int} M$, that implies the stability of the solution $u$.
\end{remark}

We stress that, as already claimed, the second result we present in this section is based on some potential theoretic tools. The first notion we need is Neumann-counterpart of the Dirichlet parabolicity. We say that a connected weighted Riemannian manifold $M_\Psi$ with (possibly empty) boundary $\partial M$ is \textit{Neumann parabolic} (or $\mathcal{N}$-\textit{parabolic}) if for any given $u\in C^0(M)\cap W^{1,2}\loc(\textnormal{int} M, \textnormal{dv}_\Psi)$ satisfying
\begin{align*}
\system{ll}{\Delta_\Psi u \geq 0 & \inn \textnormal{int}M \\ \partial_{\nu}u\leq 0 & \onn \partial M \\ \sup_M u<+\infty}
\end{align*}
it holds
\begin{align*}
u\equiv const.,
\end{align*}
where $\nu$ is the outward pointing unit normal to $\partial M$. In the case $\partial M=\emptyset$, the normal derivative condition is void.
\bigskip

As an application of Theorem \ref{Thm:UnboundedMaximumPrinciple}, we can replace \eqref{Eq:ConditionB} in \cite[Theorem 6.5]{bisterzo2022symmetry} with the (simpler) strong stability condition of $u$. Moreover, we only need the manifold $N_\Gamma$ to be parabolic, avoiding the assumption on the finiteness of its volume (originally required in \cite{bisterzo2022symmetry}).

\begin{theorem}\label{Thm:Symmetry2}
Let $M_\Psi=(I\times_\sigma N)_\Psi$ where $(N,g^N)$ is a complete (possibly noncompact), connected, $(n-1)$-dimensional Riemannian manifold without boundary. Moreover, assume that $N_\Gamma$ is parabolic.

Let $u\in C^4\left(\overline{A}(r_1,r_2)\right)$ be a solution of the Dirichlet problem
\begin{align*}
\system{ll}{\Delta_\Psi u=f(u) & \inn A(r_1,r_2) \\ u\equiv c_1 & \onn \{r_1\}\times N \\  u\equiv c_2 & \onn \{r_2\}\times N }
\end{align*}
where $c_j\in \rr$ are given constants and the function $f(t)$ is of class $C^2$ and satisfies $f''(t)\leq 0$. If $u$ is strongly stable and
\begin{align*}
\norm{u}{C^2_{rad}}=\sup_{A(r_1,r_2)} |u|+\sup_{A(r_1,r_2)} |\partial_r u|+\sup_{A(r_1,r_2)} |\partial_r^2 u|<+\infty,
\end{align*}
then $u(r,\xi)=\widehat{u}(r)$ is symmetric.
\end{theorem}
%

\begin{proof}[Proof of Theorem \ref{Thm:Symmetry2}]
Let us consider the function 
\begin{align*}
v(r,\xi):=\Delta^N_{\Gamma} u(r,\xi)
\end{align*}
which vanishes on $\partial A(r_1,r_2)$. By a direct calculation we have $[\Delta^M_\Psi, \Delta^N_\Gamma]=0$, that implies
\begin{align*}
\Delta^M_\Psi v&=\Delta^N_\Gamma f(u)\\
&=f''(u) |\nabla^N u|^2_N+f'(u)v\\
&\leq f'(u) v.
\end{align*}
It follows that $v$ satisfies
\begin{align*}
\system{ll}{\Delta_\Psi(-v)\geq f'(u) (-v) & \inn A(r_1,r_2) \\ -v=0 & \onn \partial A(r_1,r_2)}.
\end{align*}
and, using the strong stability assumption on $u$, by Theorem \ref{Thm:UnboundedMaximumPrinciple} we get
\begin{align}\label{Eq:Symmetry2App1}
v\geq 0 \ \inn A(r_1,r_2).
\end{align}
On the other hand, thanks to the parabolicity of $N_\Gamma$, we can apply \cite[Proposition 3.1]{GRIGORYAN2013607} and \cite[Lemma 6.12]{bisterzo2022symmetry} obtaining
\begin{align*}
\int_{A(r_1,r_2)} v\ \textnormal{dv}_\Psi & = \int_{r_1}^{r_2} \left(\int_{\{t\}\times N} \Delta^N_\Gamma u (t,\xi)\ \textnormal{dv}_\Gamma(\xi) \right) e^{-\Phi(t)} \sigma^{m-1}(t) \dint{t}=0
\end{align*}
that, together with \eqref{Eq:Symmetry2App1}, implies $v\equiv 0$ in $A(r_1,r_2)$.

It follows that for every fixed $\overline{r}\in [r_1,r_2]$ the function $\xi\mapsto v(\overline{r},\xi)$ is constant on $N$ and thus $\xi\mapsto u(\overline{r},\xi)$ is a bounded harmonic function on the parabolic manifold $N_\Gamma$. By definition of parabolicity, this implies that $u(\overline{r},\cdot)$ is constant in $N_\Gamma$, as claimed.
\end{proof}

\section*{Acknowledgements}
\noindent The author would like to thank Stefano Pigola for the several discussions and precious suggestions about the present work, Giona Veronelli for engaging in productive conversations about the ABP inequality and Alberto Farina for having introduced the author to the Euclidean results that inspired Section \ref{Sec:Maximum principle and cone}. The author acknowledges the support of the GNAMPA (INdAM) project ``Applicazioni geometriche del metodo ABP''.

\bibliographystyle{abbrv}
\bibliography{references}

\end{document}